\documentclass[11pt,reqno]{amsart}

\usepackage{amsmath, amsthm, amssymb}
\usepackage{amsfonts}
\usepackage[ansinew]{inputenc}
\usepackage[dvips]{epsfig}
\usepackage[english]{babel}
\usepackage{mathptmx}

\usepackage{ifthen}

\usepackage{hyperref}
\hypersetup{
	colorlinks   = true,
	linkcolor    = black,
	citecolor    = gray
}

\usepackage{cite}
\usepackage{graphicx}
\usepackage{amscd}
\usepackage[dvipsnames]{xcolor}
\usepackage{bm}
\usepackage{enumerate}

\usepackage{verbatim}
\usepackage{hyperref}
\usepackage{amstext}
\usepackage{latexsym}
 \usepackage{mathrsfs}

\theoremstyle{plain}
\newtheorem{theorem}{Theorem}[section]

\newtheorem{corollary}[theorem]{Corollary}
\newtheorem{proposition}[theorem]{Proposition}
\newtheorem{lemma}[theorem]{Lemma}
\newtheorem{remark}[theorem]{Remark}

\numberwithin{theorem}{section}
\numberwithin{equation}{section}

\def\R{\mathbb{R}}

\def\div{\text{div}}

\renewcommand{\a }{\alpha }
\renewcommand{\b }{\beta }
\renewcommand{\d}{\delta }

\newcommand{\D }{\Delta }

\newcommand{\G }{\Gamma}
\renewcommand{\l }{\lambda }

\newcommand{\n }{\nabla }

\newcommand{\s }{\sigma }

\renewcommand{\t }{\tau }

\renewcommand{\o }{\omega }
\renewcommand{\O }{\Omega }

\newcommand{\ov}{\overline}

\newcommand{\be}{\begin{equation}}
\newcommand{\ee}{\end{equation}}

\newcommand{\ti}{\widetilde}

\newcommand{\al}{\alpha}
\renewcommand{\k}{\kappa}

\newcommand{\calH }{\mathcal{H}}

\newcommand{\calE }{\mathcal{E}}
\newcommand{\calL }{\mathcal{L}}

\newcommand{\N}{\mathbb{N}}

\newcommand{\cH}{{\mathcal H}}

\newcommand{\cL}{{\mathcal L}}

\newcommand{\supp}{{\rm supp}}

\newcommand{\eps}{\varepsilon}

\renewcommand{\epsilon}{\varepsilon}

\newcommand{\Ds}{ (-\D)^s}

\title{Nonradiality of second fractional eigenfunctions of thin annuli}

\author[]
{Sidy M. Djitte and Sven Jarohs}

\address{African Institute for Mathematical Sciences in Senegal (AIMS Senegal), 
	KM2, Route de Joal, B.P. 14 18. Mbour, S\'en\'egal.}

\email{sidy.m.djitte@aims-senegal.org}

\address{Goethe-Universit\"{a}t Frankfurt, Institut f\"{u}r Mathematik.
	Robert-Mayer-Str. 10, D-60629 Frankfurt, Germany.}

\email{jarohs@math.uni-frankfurt.de}

\thanks{\textit{MSC2020}:
\textit{Primary:}
47A75, 
47G30; 
\textit{Secondary:}
35B06, 
}

\begin{document}

 \begin{abstract} In the present paper, we study properties of the second Dirichlet eigenvalue of the fractional Laplacian of annuli-like domains and the corresponding eigenfunctions. In the first part, we consider an annulus with inner radius $R$ and outer radius $R+1$. We show that for $R$ sufficiently large any corresponding second eigenfunction of this annulus is nonradial. In the second part, we investigate the second eigenvalue in domains of the form $B_1(0)\setminus \overline{B_{\tau}(a)}$, where $a$ is in the unitary ball and $0<\tau<1-|a|$. We show that this value is maximized for $a=0$, if the set $B_1(0)\setminus \overline{B_{\tau}(0)}$ has no radial second eigenfunction. We emphasize that the first part of our paper implies that this assumption is indeed nonempty.
   \noindent 
 \end{abstract}

\maketitle

\section{introduction}
Let $\O$ be a radial open bounded subset of $\R^N$, $N\geq 2$. In the first part of this note, we are interested in symmetry properties of weak solutions to the eigenvalue problem 
\be\label{eq-eigenvalue-problem}
\Ds\phi = \mu\phi\quad\text{in}\quad\O\qquad\text{and}\qquad \phi=0\quad\text{in}\quad\R^N\setminus\O,
\ee
where $\Ds$ is the fractional Laplace operator,
which is defined, for $\phi\in C^\infty_c(\R^N)$ by 
$$
\Ds\phi(x)=\frac{b_{N,s}}{2}\int_{\R^N}\frac{2\phi(x)-\phi(x+y)-\phi(x-y)}{|y|^{N+2s}}dy,\qquad x\in \R^N
$$
with $b_{N,s}=\frac{s4^s\G(\frac{N}{2}+s)}{\pi^{\frac{N}{2}}\G(1-s)}$. Here, a function $\phi$ is called a weak solution of \eqref{eq-eigenvalue-problem}, if $\phi\in\calH^s_0(\O)$ and for all $\psi\in\calH^s_0(\O)$ it holds that 
$$
\calE_s(\phi,\psi):=\frac{b_{N,s}}{2}\int_{\R^{2N}}\frac{(\phi(x)-\phi(y))(\psi(x)-\psi(y))}{|x-y|^{N+2s}}dxdy=\mu\int_{\O}\phi(x)\psi(x)dx.
$$
As usual, $\cH^s_0(\O)$ is defined as the completion of $C^\infty_c(\O)$ with respect to the norm $\|\psi\|^2_{H^s(\R^N)}:=\calE_s(\psi,\psi)$. Recall here, that since $\O$ is bounded, it follows that $\calE_s$ is a scalar product on $\cH^s_0(\O)$. Recall moreover that if $\O$ has a continuous boundary, then the space $\calH^s_0(\O)$ coincides also with the space $\{v\in L^2_{loc}(\R^N):\calE_s(v,v)<\infty\;\;\text{and}\;\; v=0\quad\text{in}\quad\R^N\setminus\O\}$. We refer to \cite{PG,Vetal} for more details and for more information about fractional Sobolev spaces. By standard theory it follows that there is an increasing sequence of real numbers $0<\l_{1,s}(\O)<\l_{2,s}(\O)\leq\ldots\leq\l_{j,s}(\O)\leq\ldots\nearrow +\infty$ such that for each $\l_{j,s}(\O)$ the equation \eqref{eq-eigenvalue-problem} has a nontrivial solution $\phi_j$ and that the first eigenvalue $\l_{1,s}(\O)$ is simple, i.e, the corresponding solution $\phi_1$ is unique up to a multiplicative constant and, moreover, can be chosen to be positive. As a consequence of the latter, it is known that $\phi_1$ always inherits the symmetry properties of the underlying domain $\O$. In particular when $\O$ is radial, one obtains that $\phi_1$ has to be radial. However for $j\geq 2$, since simplicity fails in general, it is a nontrivial task to decide whether
the corresponding solutions would inherit symmetry properties of the domain $\O$ or not.\\
A conjecture by Ba\~nuelos and Kulczycki (see \cite{Dyda}) states that when $\O$ is a ball and $j=2$, then a solution corresponding to \eqref{eq-eigenvalue-problem} with $\mu=\l_{2,s}(\O)$ cannot be radial. Several partial answers were obtained in \cite{Betal,RK,Dyda,Fereira,MK} and it is only recently that the conjecture is fully solved in \cite{Fetal} by estimating the Morse index of a radial eigenfunction (see also \cite{Betal}, with a different approach to prove this conjecture). The study of the Morse index of radial functions was in particular studied in \cite{AP04} for the Laplacian, that is the case $s=1$, in balls and annuli. In \cite{Fetal} this approach has been extended to the nonlocal framework in balls. Our first result concerns the extension of the Ba\~nuelos-Kulczycki conjecture to annuli and is in the spirit of \cite{AP04,Fetal} in annuli. We show the following for $A_R:=\{x\in\R^N: R<|x|<R+1\}$ with $R>0$.
\begin{theorem}\label{eq-main-thm}
There is $R_0>0$ such that for any $R\geq R_0$ any second eigenfunction corresponding to $\l_{2,s}(A_R)$ is nonradial.
\end{theorem}
Note that due to the scaling properties of the fractional Laplacian we immediately deduce the following corollary to this result.
\begin{corollary}\label{eq-main-cor}
There is $\t_0\in(0,1)$ such that for any $\tau\in[\t_0,1)$ any second eigenfunction corresponding to $\l_{2,s}(B_1(0)\setminus B_{\t}(0))$ is nonradial.
\end{corollary}
To prove Theorem \ref{eq-main-thm}, the main point is to establish that any second radial eigenfunction $u_R$ of $A_R$ satisfies $\frac{u_R}{d^s}(R)\frac{u_R}{d^s}(R+1)<0$ for $R$ sufficiently large. Here, $d$ denotes the distance function to the boundary of $A_R$. Once we have this, one can argue exactly as in \cite{Fetal} to conclude the proof of the Theorem. As already mentioned in \cite{Fetal}, such property cannot be proved by using the classical Hopf lemma since $u_R$ is sign changing; and it also does not follow from the fractional Pohozaev identity (see e.g \cite{RS}) either. To obtain the latter, we use a compactness argument to show that, along some subsequence, $\psi_R:=\frac{R^{\frac{N-1}{2}}u_R(\cdot+R)}{d^s}\rightarrow \frac{\phi_2}{d^s}$  in $C^0([0,1])$ where $\phi_2$ is a second eigenfunction of the unit interval $(0,1)$. From there we deduce the claim since the limiting problem has already the desired property by the result of \cite{Fetal}.\\
We note that indeed we expect the conclusion of Corollary \ref{eq-main-cor} to hold for any $\t\in(0,1)$, and thus we only give here a partial answer. Moreover, though Theorem \ref{eq-main-thm} and Corollary \ref{eq-main-cor} are stated for $N\geq 2$, the same can be shown for $N=1$ with a similar but simpler proof to the one we present here.\\

Our next result concerns the maximization of $\l_{2,s}$ of certain \textit{shifted} annuli. To be precise, given $\t\in(0,1)$ we aim at finding
$$
\sup_{a\in B_{1-\t}(0)}\l_{2,s}(B_1(0)\setminus B_\t(a)).
$$
For the classical case of the Laplacian the corresponding problem was studied in \cite{SK}, where it was proven that concentric spheres maximize the second eigenvalue, i.e. $\sup_{a\in B_{1-\tau}(0)}\l_{2,1}(B_1(0)\setminus B_\tau(a))= \l_{2,1}(B_1(0)\setminus B_{\t}(0))$. We show the following.

\begin{theorem}\label{thm:max-prob} Let $\tau\in(0,1)$ and assume any second eigenfunction of the annulus $B_1(0)\setminus\ov{B_\t(0)}$ cannot be radial. Then 
\be\label{max-prob}
\l_{2,s}\big(B_1(0)\setminus\ov{B_\t(a)}\big)\leq \l_{2,s}\big(B_1(0)\setminus\ov{B_\t(0)}\big) \quad\text{for all $a\in B_{1-\t}(0)$}
\ee
and equality holds in \eqref{max-prob} if and only if $a=0$.
\end{theorem}
Note that by Corollary \ref{eq-main-cor}, the assumption of Theorem \ref{thm:max-prob} is satisfied when $\t$ is sufficiently close to $1$. To prove Theorem \ref{thm:max-prob}, a key observation is that the second eigenvalue of an eccentric annulus is always controlled by its first antisymmetric eigenvalue and that for an annulus the two numbers coincide. These two properties allow to reduce the proof of \eqref{max-prob} into proving that the first antisymmetric eigenvalue of an eccentric annulus decreases when the obstacle (the inner ball) moves from the center to the boundary of the unitary ball. To get the latter, we use a shape derivative argument combined with the maximum principle for doubly antisymmetric functions that we established in \cite{SS}. We refer to Section \ref{sec-6} for more details.\\

This paper is organized as follows: Section \ref{sec-2} and \ref{sec-3} contain preliminary results that will be used later in Section \ref{sec-4} to obtain uniform H\"older estimates of the fractional normal derivative. Section \ref{sec-5} is devoted to the proof Theorem \ref{eq-main-thm} and in the last section we prove Theorem \ref{thm:max-prob}.\\

\textbf{Acknowledgements.} This work is supported by DAAD and BMBF (Germany) within the project 57385104. The authors thank Mouhamed Moustapha Fall and Tobias Weth for helpful discussions. We also thank the anonymous referee for useful suggestions that improved the paper.

\section{Preliminary results}\label{sec-2}

Let $\Gamma(z)=\int_0^{\infty}t^{z-1}e^{-t}\ dt$, $z>0$ be the Gamma-function. In the following, we use strongly the identity
$$
\int_0^{\infty}\frac{h^{a-1}}{(1+h)^{a+b}}\ dh=\frac{\Gamma(a)\Gamma(b)}{\Gamma(a+b)}\quad\text{for $a,b>0$.}
$$
\begin{lemma}\label{eq-first-lemma}
Let $N\in \N$, $N\geq 2$, $s\in(0,1)$, and consider the function
$$ 
\psi:[0,1]\to \R,\quad\psi(t)=\left\{\begin{aligned} &t^{1+2s}\int_0^1\frac{(h(1-h))^{\frac{N-3}{2}}}{(t^2+h)^{\frac{N+2s}{2}}}\ dh, && t>0;\\
&\frac{\Gamma(\frac{1}{2}+s)\Gamma(\frac{N-1}{2})}{\Gamma(\frac{N+2s}{2})},&&t=0.\end{aligned}\right.
$$
Then the following holds.
\begin{enumerate}
\item If $s> \frac12$, then $\psi\in C^{0,1}([0,1])$.
\item If $s\in(0,\frac12)$, then $\psi\in C^{0,2s}([0,1])$.
\item If $s=\frac{1}{2}$, then $\psi\in C^{0,\sigma}([0,1])$ for all $\sigma\in(0,1)$.
\end{enumerate}
In particular, $\psi\in C^{s+\d}([0,1])$ for some $\d=\d(s)>0$.
\end{lemma}
\begin{proof}
First note that if $N=3$, then it directly follows that $\psi(0)=\frac{\Gamma(\frac{1+2s}{2})}{\Gamma(\frac{3+2s}{2})}=\frac{\Gamma(\frac{1+2s}{2})}{\Gamma(1+\frac{1+2s}{2})}=\frac{2}{1+2s}$ and for $t>0$ we have
\begin{align*}
\psi(t)&=t^{1+2s}\int_0^1\frac{1}{(t^2+h)^{\frac{3+2s}{2}}}\ dh=\frac{2t^{1+2s}}{1+2s}\Big(\frac{1}{t^{1+2s}}-\frac{1}{(1+t^2)^{\frac{1+2s}{2}}}\Big)=\frac{2}{1+2s}\Big(1-\frac{t^{1+2s}}{(1+t^2)^{\frac{1+2s}{2}}}\Big),
\end{align*}
so that it follows easily that $\psi\in C^{0,1}([0,1])$.\\
In the following let $N\neq3$. We begin by transforming the integral slightly. Note that for $t>0$ and with the substitution $h=f(\tau):=\frac{\tau}{1+\tau}$ we find
\begin{align*}
\psi(t)&=t^{1+2s}\int_0^1\frac{(h(1-h))^{\frac{N-3}{2}}}{(t^2+h)^{\frac{N+2s}{2}}}\ dh=t^{1+2s}\int_0^{\infty}\frac{f'(\tau) (f(\tau)(1-f(\tau)))^{\frac{N-3}{2}}}{(t^2+f(\tau))^{\frac{N+2s}{2}}}\ d\tau\\
&=t^{1+2s}\int_0^{\infty}\frac{ \tau^{\frac{N-3}{2}}}{(1+\tau)^{N-1}(t^2+\frac{\tau}{1+\tau})^{\frac{N+2s}{2}}}\ d\tau\\
&=t^{1+2s}\int_0^{\infty}\frac{ \tau^{\frac{N-3}{2}}}{(1+\tau)^{\frac{N}{2}-1-s}(t^2+(1+t^2)\tau)^{\frac{N+2s}{2}}}\ d\tau\\
&=\frac{t^{1+2s}}{(1+t^2)^{\frac{N+2s}{2}}}\int_0^{\infty}\frac{ \tau^{\frac{N-3}{2}}}{(1+\tau)^{\frac{N}{2}-1-s}(\frac{t^2}{1+t^2}+\tau)^{\frac{N+2s}{2}}}\ d\tau\\
&=(f^{-1}(f(t^2)))^{\frac{1-N}{2}}f(t^2)^{\frac{N+2s}{2}}\int_0^{\infty}\frac{ \tau^{\frac{N-3}{2}}}{(1+\tau)^{\frac{N}{2}-1-s}(f(t^2)+\tau)^{\frac{N+2s}{2}}}\ d\tau\\
&=\Big(\frac{f(t^2)}{1-f(t^2)}\Big)^{\frac{1-N}{2}}f(t^2)^{\frac{N+2s}{2}}\int_0^{\infty}\frac{ \tau^{\frac{N-3}{2}}}{(1+\tau)^{\frac{N}{2}-1-s}(f(t^2)+\tau)^{\frac{N+2s}{2}}}\ d\tau\\
&=(1-T)^{\frac{N-1}{2}}T^{\frac{1+2s}{2}}\int_0^{\infty}\frac{ \tau^{\frac{N-3}{2}}}{(1+\tau)^{\frac{N}{2}-1-s}(T+\tau)^{\frac{N+2s}{2}}}\ d\tau,
\end{align*}
where we put $T=f(t^2)$ and used that $f^{-1}(a)=\frac{a}{1-a}$ for $a\in[0,1)$. Hence
\begin{align*}
\psi(t)&=(1-T)^{\frac{N-1}{2}}T^{\frac{1+2s}{2}}\int_0^{\infty}\frac{ \tau^{\frac{N-3}{2}}}{(1+\tau)^{\frac{N}{2}-1-s}(T+\tau)^{\frac{N+2s}{2}}}\ d\tau\\
&=(1-T)^{\frac{N-1}{2}}T^{\frac{1+2s}{2}}\int_0^{\infty}\frac{ T^{\frac{N-1}{2}}h^{\frac{N-3}{2}}}{(1+Th)^{\frac{N}{2}-1-s}(T+Th)^{\frac{N+2s}{2}}}\ dh\\
&=(1-T)^{\frac{N-1}{2}}\int_0^{\infty}\frac{h^{\frac{N-3}{2}}}{(1+Th)^{\frac{N}{2}-1-s}(1+h)^{\frac{N+2s}{2}}}\ dh=:\Psi(T),
\end{align*}
that is, $\Psi(T)=\psi(\sqrt{f^{-1}(T)})$ for $T\in(0,\frac{1}{2}]$. Moreover, note that
\begin{equation}
\label{generalbound}
\begin{split}
(1-T)^{\frac{N-1}{2}}\int_0^{\infty}\frac{h^{\frac{N-3}{2}}}{(1+Th)^{\frac{N}{2}-1-s}(1+h)^{\frac{N+2s}{2}}}\ dh&\leq \int_0^{\infty}\frac{h^{\frac{N-3}{2}}}{(1+Th)^{\frac{N}{2}-1}(1+h)^{\frac{N}{2}}}\Big(\frac{1+Th}{1+h}\Big)^{s}\ dh\\
&\leq \int_0^{\infty}\frac{h^{\frac{N-3}{2}}}{(1+h)^{\frac{N}{2}}}\ dh=\frac{\sqrt{\pi}\Gamma(\frac{N-1}{2})}{\Gamma(\frac{N}{2})}=:C_1
\end{split}
\end{equation}
Thus, by dominated convergence and the above, we find
$$
\Psi(0):=\lim_{t\to0}\Psi(T)= \int_0^{\infty}\frac{h^{\frac{N-3}{2}}}{(1+h)^{\frac{N+2s}{2}}}\ dh=\frac{\Gamma(\frac12+s)\Gamma(\frac{N-1}{2})}{\Gamma(\frac{N+2s}{2})}=\psi(0).
$$
Hence, $\Psi:[0,\frac{1}{2}]\to\R$ is continuous and thus it follows that also $\psi$ is continuous. To show the H\"older continuity of $\psi$, we use the representation via $\Psi$ and consider different cases.\\
%

\noindent\textit{The case $s>\frac12$:} Let $0\leq A<B\leq \frac12$. Then
\begin{align*}
|\Psi(B)-\Psi(A)|&\leq (1-B)^{\frac{N-1}{2}}\int_0^{\infty}\frac{h^{\frac{N-3}{2}}}{(1+h)^{\frac{N+2s}{2}}}\Bigg|\frac{1}{(1+Bh)^{\frac{N}{2}-1-s}}-\frac{1}{(1+Ah)^{\frac{N}{2}-1-s}}\Bigg|\ dh\\
&\qquad +\int_0^{\infty}\frac{h^{\frac{N-3}{2}}}{(1+ah)^{\frac{N}{2}-1-s}(1+h)^{\frac{N+2s}{2}}}\ dh\Bigg|(1-B)^{\frac{N-1}{2}}-(1-A)^{\frac{N-1}{2}}\Bigg|\\
&\leq |\frac{N}{2}-1-s| \int_0^{\infty}\frac{h^{\frac{N-1}{2}}}{(1+h)^{\frac{N+2s}{2}}}\sup_{x\in[A,B]}\frac{1}{(1+xh)^{\frac{N}{2}-s}}\ dh|A-B|\\
&\qquad +\frac{N-1}{2}\int_0^{\infty}\frac{h^{\frac{N-3}{2}}}{(1+ah)^{\frac{N}{2}-1-s}(1+h)^{\frac{N+2s}{2}}}\ dh\sup_{x\in[A,B]}(1-x)^{\frac{N-3}{2}}|A-B|.
\end{align*}
Clearly, using \eqref{generalbound}, 
$$
\frac{N-1}{2}\int_0^{\infty}\frac{h^{\frac{N-3}{2}}}{(1+Ah)^{\frac{N}{2}-1-s}(1+h)^{\frac{N+2s}{2}}}\ dh\sup_{x\in[A,B]}(1-x)^{\frac{N-3}{2}}\leq \frac{\sqrt{\pi}\Gamma(\frac{N+1}{2}) }{\Gamma(\frac{N}{2})}=C_1.
$$
Moreover, if $\frac{N}{2}\geq s$ we have $\sup_{x\in[A,B]}\frac{1}{(1+xh)^{\frac{N}{2}-s}}=1$ and thus also
\begin{align*}
|\frac{N}{2}-1-s| &\int_0^{\infty}\frac{h^{\frac{N-1}{2}}}{(1+h)^{\frac{N+2s}{2}}}\sup_{x\in[A,B]}\frac{1}{(1+xh)^{\frac{N}{2}-s}}\ dh\\
&\leq|\frac{N}{2}-1-s|\int_0^{\infty}\frac{h^{\frac{N-1}{2}}}{(1+h)^{\frac{N}{2}+s}}\ dh=|\frac{N}{2}-1-s|\frac{\Gamma(s-\frac{1}{2})\Gamma(\frac{N+1}{2})}{\Gamma(\frac{N+2s}{2})}=:C_{2}.
\end{align*}
Hence, $\Psi$ is Lipschitz continuous. Thus, for $0\leq a<b\leq 1$ we find with $C_3=\max\{C_1,C_2\}$
\begin{align*}
|\psi(a)-\psi(b)|&=\Big|\Psi(\frac{a^2}{1+a^2})-\Psi(\frac{b^2}{1+b^2})\Big|\leq C_3\Big|\frac{a^2}{1+a^2}-\frac{b^2}{1+b^2}\Big|\\
&\leq C_3\sup_{x\in[a^2,b^2]}|\frac{2x}{(1+x^2)^2}||a^2-b^2|\leq C_3|a-b|,
\end{align*}
using that $a,b\leq 1$. Hence, $\psi$ is Lipschitz continuous in this case.\\

\noindent\textit{The case $s<\frac12$:} We use the inequality
$$
|x^{2s-N}-y^{2s-N}|\leq c|x-y|^{2s}\Big(x^{-2-N}+y^{-2-N}\Big)\quad \text{for }\ x,y>0
$$
for a constant $c=c_{s,N}>0$.
Let $0\leq A<B\leq \frac12$. Proceeding as in the previous case, we find with the above inequality
\begin{align*}
&|\Psi(B)-\Psi(A)|\leq \int_0^{\infty}\frac{h^{\frac{N-3}{2}}}{(1+h)^{\frac{N+2s}{2}}}\Bigg| (1+Bh)^{s+1-\frac{N}{2}}-(1+Ah)^{s+1-\frac{N}{2}}\Bigg|\ dh+C_1|A-B|\\
&\leq C_1|A-B|+c|A-B|^{2s}\int_0^{\infty}\frac{h^{\frac{N-3}{2}+2s}}{(1+h)^{\frac{N+2s}{2}}}\max\{(1+Ah)^{1-s-\frac{N}{2}},(1+Bh)^{1-s-\frac{N}{2}}\}\ dh\\
&\leq C_1|A-B|+c|A-B|^{2s}\int_0^{\infty}\frac{h^{\frac{N-3}{2}+2s}}{(1+h)^{\frac{N+2s}{2}}}\ dh\\
&=\underbrace{\Big(C_1+c\frac{\Gamma(\frac12-s)\Gamma(\frac{N-1}{2}+2s)}{\Gamma(\frac{N+2s}{2})}\Big)}_{=:C_4}|A-B|^{2s}.
\end{align*}
Similarly as in the previous case we find
$$
|\psi(a)-\psi(b)|=\Big|\Psi(\frac{a^2}{1+a^2})-\Psi(\frac{b^2}{1+b^2})|\Big|\leq C_4|a^2-b^2|^{2s}\leq C_4|a-b|^{2s}.
$$
This shows the case for $s<\frac12$.\\

\noindent\textit{The case $s=\frac12$:} Let $\sigma\in(0,1)$. Then we have for $0<x<y$ by H\"older's inequality with $\frac{1}{p}=\sigma$, $\frac{1}{q}=1-\sigma$
\begin{align*}
\Big|x^{\frac{3-N}{2}}-y^{\frac{3-N}{2}}\Big|&=\Bigg|\frac{3-N}{2}\int_x^y t^{\frac{1-N}{2}}\ dt\Bigg|\leq \Big|\frac{3-N}{2}\Big||x-y|^{\sigma}\Big(\int_{x}^y t^{q(\frac{1-N}{2})}\ dt\Big)^{1-\sigma}\\
&\leq \frac{|N-3|}{2(q(\frac{N-1}{2})+1)^{1-\sigma}}|x-y|^{\sigma}\max\Big\{x^{\frac{3-N}{2}-\sigma},y^{\frac{3-N}{2}-\sigma}\Big\}\\
&= \underbrace{\frac{|N-3|}{2(\frac{N+1}{2}-\sigma)^{1-\sigma}}(1-\sigma)^{1-\sigma}}_{=:C_5}|x-y|^{\sigma}\max\Big\{x^{\frac{3-N}{2}-\sigma},y^{\frac{3-N}{2}-\sigma}\Big\}.
\end{align*}
Then we find similar to the previous case for $0\leq A<B\leq \frac12$
\begin{align*}
|\Psi(B)-\Psi(A)|&\leq \int_0^{\infty}\frac{h^{\frac{N-3}{2}}}{(1+h)^{\frac{N+1}{2}}}\Bigg| (1+Bh)^{\frac{3-N}{2}}-(1+Ah)^{\frac{3-N}{2}}\Bigg|\ dh+C_1|A-B|\\
&\leq C_1|A-B|+C_5|A-B|^{\sigma}\int_0^{\infty}\frac{h^{\frac{N-3}{2}+\sigma}}{(1+h)^{\frac{N+1}{2}}}\max\{(1+Ah)^{\frac{3-N}{2}-\sigma},(1+Bh)^{\frac{3-N}{2}-\sigma}\}\ dh\\
&\leq C_1|A-B|+C_5|A-B|^{\sigma}\int_0^{\infty}\frac{h^{\frac{N-3}{2}+\sigma}}{(1+h)^{\frac{N+1}{2}}}\ dh\\
&=\Big(C_1+C_5\frac{\Gamma(1-\sigma)\Gamma(\frac{N-1}{2}+\sigma)}{\Gamma(\frac{N+1}{2})}\Big)|A-B|^{\sigma}.
\end{align*}
As before we conclude that $\psi \in C^{0,\sigma}([0,1])$. This finishes the proof.
\end{proof}
As a corollary we have the following 
\begin{corollary}\label{eq-Holder-reg-2}
Let $N\geq 2$ and $R>1$. Define $$F_R^{\pm}:[-2,2]\times[0,2]\to\R_+,\;\;\;\;(t,r)\mapsto F_R^{\pm}(t,r)=\left\{\begin{aligned} &\int_{\frac{\sqrt{(t+R)(t\pm r+R)}}{r}(\mathbb{S}^{N-1}-e_1)}\frac{d\theta}{(1+|\theta|^2)^{\frac{N+2s}{2}}}, && \forall\,r\neq 0;\\
&\frac{\pi^{\frac{N-1}{2}}\G(\frac{1+2s}{2})}{\G(\frac{N+2s}{2})},&&\text{for}\quad r=0.\end{aligned}\right.
$$
Then, we have 
\begin{equation}\label{eq-new-def-of-F_R}
    F_R^{\pm}(r,t)=\frac{\pi^{\frac{N-1}{2}}}{\G(\frac{N-1}{2})}\psi\Big(\frac{r}{2\sqrt{(t+R)(t\pm r+R)}}\Big)\qquad\forall\,\,t\in[-2,2],\,\forall\,\,r\in[0,2],
\end{equation}
where $\psi$ is defined as in Lemma \ref{eq-first-lemma} above. Consequently, there exists $C,\delta>0$ (independent of $R$) so that 
\begin{equation}\label{eq-Holder-regularity-of-F_R}
|F_R^{\pm}(t,r)-F_R^{\pm}(t',r')|\leq C(|r-r'|^{s+\d}+|t-t'|^{s+\d}),\quad\forall\,\,r,r'\in[0,2]\quad\text{and}\quad\forall\,\,t,t'\in[-2,+2].
\end{equation}
\end{corollary}
\begin{proof}
First, note that the estimate \eqref{eq-Holder-regularity-of-F_R} is a simple consequence of the identity \eqref{eq-new-def-of-F_R} and Lemma \ref{eq-first-lemma} above. Therefore, the proof of the Corollary reduces to the proof of the identity \eqref{eq-new-def-of-F_R}. By definition of $\psi$, we immediately check \eqref{eq-new-def-of-F_R} for $r=0$. For $r\neq 0$, we let 
$$
\t_R(t,r):=\frac{\sqrt{(t+R)(t\pm r+R)}}{r}.
$$
Since the integrand is invariant under rotation, by changing variables $\ov y = y+\t_R(t,r) e_N$ we get 
\begin{align*}
F_R^{\pm}(t,r)&=\int_{\t_R(t,r)(\mathbb{S}^{N-1}-e_N)}\frac{dy}{(1+|y|^2)^{\frac{N+2s}{2}}}= \int_{\t_R(t,r)\mathbb{S}^{N-1}}\frac{dy}{(1+|y-\t_R(t,r) e_N|^2)^{\frac{N+2s}{2}}}\\
&=\int_{\t_R(t,r)\mathbb{S}^{N-1}}\frac{dy}{(1+|y|^2-2\t_R(t,r) y\cdot e_N+\t_R^2(t,r))^{\frac{N+2s}{2}}}\\
&=\int_{\t_R(t,r)\mathbb{S}^{N-1}}\frac{dy}{(1+2\t_R^2(t,r)-2\t_R(t,r)y\cdot e_N)^{\frac{N+2s}{2}}}.
\end{align*}
We first start with the case $N\geq 3$. Passing into spherical coordinates we get
\begin{align*}
F_R^{\pm}(t,r)&=\int_{(0,\pi)^{N-2}}d\theta_1\cdots d\theta_{N-2}\t_R^{N-1}(t,r)\sin^{N-2}(\theta_1)\sin^{N-3}(\theta_2)\cdots\sin(\theta_{N-2})\\
&\times \int_{0}^{2\pi} \frac{d\varphi}{(1+2\t_R^2(t,r)-2\t_R^2(t,r)\cos\theta_1)^{\frac{N+2s}{2}}}\\
&=2\pi c_N\int_{0}^\pi\frac{\t_R^{N-1}(t,r)\sin^{N-2}(\theta_1)}{(1+2\t_R^2(t,r)-2\t_R^2(t,r)\cos\theta_1)^{\frac{N+2s}{2}}},
\end{align*}
where $c_N=\int_{(0,\pi)^{N-3}}\sin^{N-3}(\theta_2)\cdots\sin(\theta_{N-2})d\theta_2\cdots d\theta_{N-2}=\frac{\pi^{\frac{N-1}{2}}}{\pi\G(\frac{N-1}{2})}$ for $N\geq 3$. Using the change of variables $h=\cos\theta_1$, we obtain
\begin{align*}
F_R^{\pm}(t,r)&=2\pi c_N\int_{-1}^{+1}\frac{\t_R^{N-1}(r,t)(1-h^2)^{\frac{N-2}{2}}}{\Big(1+2\t_R^2(t,r)-2\t_R^2(t,r)h\Big)^{\frac{N+2s}{2}}}\frac{dh}{(1-h^2)^{1/2}}\\
&=2\pi c_N\t_R^{N-1}(t,r)\int_{-1}^{+1}\frac{\Big(1-h^2)^{\frac{N-3}{2}}}{(1+2\t_R^2(t,r)-2\t_R^2(t,r)h\Big)^{\frac{N+2s}{2}}}dh\\
&=2\pi c_N\t_R^{N-1}(t,r)\int_{0}^{2}\frac{h^{\frac{N-3}{2}}(2-h)^{\frac{N-3}{2}}}{(1+2\t_R^2(t,r)h)^{\frac{N+2s}{2}}}dh\\
&=2\pi c_N\t_R^{-1-2s}(t,r)\int_{0}^{2}\frac{h^{\frac{N-3}{2}}(2-h)^{\frac{N-3}{2}}}{\big(\frac{1}{\t_R^2(t,r)}+2h\big)^{\frac{N+2s}{2}}}dh\\
&=\frac{\pi^{\frac{N-1}{2}}}{\G(\frac{N-1}{2})}\big(\frac{1}{2\t_R(t,r)}\big)^{1+2s}\int_{0}^{1}\frac{h^{\frac{N-3}{2}}(1-h)^{\frac{N-3}{2}}}{\Big(\big(\frac{1}{2\t_R(t,r)}\big)^2+h\Big)^{\frac{N+2s}{2}}}dh\\
&=\frac{\pi^{\frac{N-1}{2}}}{\G(\frac{N-1}{2})}\psi\Big(\frac{1}{2\t_R(t,r)}\Big)\qquad\forall\,\,t\in[-2,+2],\,\forall\,\,r\in(0,2].
\end{align*}
%
%
The case $N=2$ is treated similarly.
\end{proof}

\section{Uniform estimates of the first and second eigenvalues of the annulus \texorpdfstring{$A_R$}{AR}.}\label{sec-3}

The aim of this section is to obtain a uniform control of the second fractional eigenvalue of the annulus $A_R=\{x\in\R^N: R<|x|<R+1\}$. For that, we let $\l_{1,s}(A_R)$ and $\l_{2,s}(A_R)$ be respectively the first and the second fractional eigenvalue of $A_R$. We start with the following 
\begin{lemma}\label{eq-uniform-estimate-lambda-1-A-R}
There exists a positive constant $C(N,s)>0$ so that 
\be
\l_{1,s}(A_R)\leq C(N,s)\quad\text{for all $R\geq 1$.}
\ee
\end{lemma}
\begin{proof}
 Let $\l_{1,s}(0,1)$ be the first eigenvalue of the interval $(0,1)$ and $\phi_1$ be the corresponding (normalized) eigenfunction. Consider $\Phi_R: x\mapsto\phi_1(|x|-R)$. It is clear that $\Phi_R=0$ in $\R^N\setminus A_R$. Moreover
\begin{align}
&R^{-(N-1)}\int_{\R^{2N}}\frac{(\Phi_R(x)-\Phi_R(y))^2}{|x-y|^{N+2s}}dxdy\nonumber\\
&=\int_{0}^\infty \int_{0}^\infty \int_{\mathbb{S}^{N-1}\times \mathbb{S}^{N-1}}\frac{r^{N-1} \tilde r^{N-1}}{R^{N-1}}\frac{\big(\phi_1(r-R)-\phi_1(\tilde r-R)\big)^2}{|r\theta-\tilde r\tilde\theta|^{N+2s}}d\theta d\tilde\theta drd\tilde r\nonumber\\
&=\int_{-R}^\infty \int_{-R}^\infty\frac{\big[(r+R)(\tilde r+R)\big]^{N-1}}{R^{N-1}} \big(\phi_1(r)-\phi_1(\tilde r)\big)^2\int_{\mathbb{S}^{N-1}}d\theta\int_{\mathbb{S}^{N-1}}\frac{d\tilde \theta}{|(r+R)\theta-(\tilde r+R)\tilde\theta|^{N+2s}} drd\tilde r\nonumber\\
&=\o_N\int_{-R}^\infty \int_{-R}^\infty\frac{\big[(r+R)(\tilde r+R)\big]^{N-1} }{R^{N-1}}\big(\phi_1(r)-\phi_1(\tilde r)\big)^2\int_{\mathbb{S}^{N-1}}\frac{d\tilde \theta}{|(r+R) e_1-(\tilde r+R)\tilde\theta|^{N+2s}} drd\tilde r.\label{eq-ar}
\end{align}
Here, $\o_N$ denotes the volume of the $(N-1)$-dimensional unit sphere $\mathbb{S}^{N-1}$. On the other hand, we have 
\begin{align}
&\int_{\mathbb{S}^{N-1}}\frac{d\theta}{|(r+R)e_1-(\tilde r+R)\theta|^{N+2s}} = \int_{\mathbb{S}^{N-1}}\frac{d\theta}{\Big((r+R)^2+(\tilde r+R)^2-2(r+R)(\tilde r+R)e_1\cdot\theta\Big)^{\frac{N+2s}{2}}}\notag\\
&=\int_{\mathbb{S}^{N-1}}\frac{d\theta}{\Big((r-\tilde r)^2+2(r+R)(\tilde r+R)(1-e_1\cdot\theta)\Big)^{(N+2s)/2}}\notag\\
&=\int_{\mathbb{S}^{N-1}}\frac{d\theta}{\Big((r-\tilde r)^2+(r+R)(\tilde r+R)|e_1-\theta|^2\Big)^{\frac{N+2s}{2}}}\notag\\
&=\frac{1}{|r-\tilde r|^{N+2s}}\int_{\mathbb{S}^{N-1}-e_1}\frac{d\theta}{\Big(1+\frac{(r+R)(\tilde r+R)}{|r-\tilde r|^2}|\theta|^2\Big)^{\frac{N+2s}{2}}}\notag\\
&=\frac{|r-\tilde r|^{-1-2s}}{(r+R)^{\frac{N-1}{2}}(\tilde r+R)^{\frac{N-1}{2}}}\int_{\frac{\sqrt{(r+R)(\tilde r+R)}}{|r-\tilde r|}(\mathbb{S}^{N-1}-e_1)}\frac{d\theta}{(1+|\theta|^2)^{\frac{N+2s}{2}}}.\label{asbasf}
\end{align}
In the following define
\be\label{eq-def-K-R}
K_R(r,\tilde r):=\frac{(r+R)^{\frac{N-1}{2}}(\tilde r+R)^{\frac{N-1}{2}}}{R^{N-1}}\int_{\frac{\sqrt{(r+R)(\tilde r+R)}}{|r-\tilde r|}(\mathbb{S}^{N-1}-e_1)}\frac{d\theta}{(1+|\theta|^2)^{\frac{N+2s}{2}}}.
\ee
Plugging \eqref{asbasf} into \eqref{eq-ar} gives 
\begin{align}
&R^{-(N-1)}\int_{\R^{2N}}\frac{(\Phi_R(x)-\Phi_R(y))^2}{|x-y|^{N+2s}}dxdy\nonumber\\
&=\o_N\int_{-R}^\infty \int_{-R}^\infty \frac{\big(\phi_1(r)-\phi_1(\tilde r)\big)^2}{|r-\tilde r|^{1+2s}}\frac{(r+R)^{\frac{N-1}{2}}(\tilde r+R)^{\frac{N-1}{2}}}{R^{N-1}}\int_{\frac{\sqrt{(r+R)(\tilde r+R)}}{|r-\tilde r|}(\mathbb{S}^{N-1}-e_1)}\frac{d\theta}{(1+|\theta|^2)^{\frac{N+2s}{2}}}\nonumber\\
&=\o_N\int_{-2}^2 \int_{-2}^2\frac{\big(\phi_1(r)-\phi_1(\tilde r)\big)^2}{|r-\tilde r|^{1+2s}}K_R(r,\tilde r)drd\tilde r+2\o_N\int_{0}^1\phi_1^2(r)\int_{(-R,\infty)\setminus(-2,2)}\frac{K_R(r,\tilde r)}{|r-\tilde r|^{1+2s}}d\tilde r.\label{eq-norm-of-rescaled-eigenfunctions}
\end{align}
By identity \eqref{eq-new-def-of-F_R}, we have 
\be\label{eq-uniform-convergence-K-R}
K_R(r,\tilde r)=\frac{(r+R)^{\frac{N-1}{2}}(\tilde r+R)^{\frac{N-1}{2}}}{R^{N-1}}\frac{\pi^{\frac{N-1}{2}}}{\G(\frac{N-1}{2})}\psi\Big(\frac{|r-\tilde r|}{2\sqrt{(r+R)(\tilde r+R)}}\Big)\rightarrow \frac{\pi^{\frac{N-1}{2}}\G(\frac{1+2s}{2})}{\G(\frac{N+2s}{2})}>1\quad\text{as}\quad R\rightarrow\infty
\ee
for all $r,\tilde r\in(-a,a)$, for any fixed $a>0$, where $\psi$ is defined as in Lemma \ref{eq-first-lemma} (see also \cite[Lemma 5.1]{CT} for a different proof). Hence, the first integral in \eqref{eq-norm-of-rescaled-eigenfunctions} is comparable to $\int_{-2}^2 \int_{-2}^2\frac{\big(\phi_1(r)-\phi_1(\tilde r)\big)^2}{|r-\tilde r|^{1+2s}}drd\tilde r$\;\, for $R$ sufficiently large. In other words,
\begin{align}\label{eq-x-1.8}
C_1(N,s)[\phi_1]^2_{H^s(-2,2)}\leq \int_{-2}^2 \int_{-2}^2\frac{\big(\phi_1(r)-\phi_1(\tilde r)\big)^2}{|r-\tilde r|^{1+2s}}K_R(r,\tilde r)drd\tilde r&\leq C_2(N,s)[\phi_1]^2_{H^s(-2,2)}
\end{align}
for $R$ sufficiently large. 
\newpage

To estimate the second integral in \eqref{eq-norm-of-rescaled-eigenfunctions}, we write 
\begin{align*}
\int_{(-R,\infty)\setminus(-2,2)}\frac{K_R(r,\tilde r)}{|r-\tilde r|^{1+2s}}d\tilde r = \int_{-R}^{-2}\cdots d\tilde r+\int_{2}^\infty\cdots d\tilde r
\end{align*}
Using the new variable $\ov r=\frac{(r+R)(R-\tilde r)}{(r+\tilde r)}$, we may write 
\begin{align*}
\int_{\frac{\sqrt{(r+R)(R-\tilde r)}}{(r+\tilde r)}(\mathbb{S}^{N-1}-e_1)}&\frac{d\theta}{(1+|\theta|^2)^{\frac{N+2s}{2}}}
=\frac{(r+R)^{\frac{N-1}{2}}(R-\tilde r)^{\frac{N-1}{2}}}{(r+\tilde r)^{N-1}}\int_{\mathbb{S}^{N-1}-e_1}\frac{d\theta}{(1+\frac{(r+R)(R-\tilde r)}{(r+\tilde r )^2}|\theta|^2)^{\frac{N+2s}{2}}}\\
&=(r+R)^{-(N-1)}\ov r^{\frac{N-1}{2}}(\ov r+r+R)^{\frac{N-1}{2}}\int_{\mathbb{S}^{N-1}-e_1}\frac{d\theta}{\Big(1+\Big|\frac{\sqrt{\ov r(\ov r+r+R)}}{R+r}\theta\Big|^2\Big)^{\frac{N+2s}{2}}}=f\big(\frac{\ov r}{r+R}\big),
\end{align*}
where we set
\be\label{eq-def-ov-f}
f(\rho)=\int_{\sqrt{\rho(\rho+1)}(\mathbb{S}^{N-1}-e_1)}\frac{d\theta}{(1+|\theta|^2)^{\frac{N+2s}{2}}}.
\ee
We also have 
\begin{align*}
 &\int_{-R}^{-2}\cdots d\tilde r=\frac{(r+R)^{\frac{N-1}{2}}}{R^{N-1}}\int_{-R}^{-2}(\tilde r+R)^{\frac{N-1}{2}}\frac{1}{|r-\tilde r|^{1+2s}}\int_{\frac{\sqrt{(r+R)(\tilde r+R)}}{|r-\tilde r|}(\mathbb{S}^{N-1}-e_1)}\frac{d\theta}{(1+|\theta|^2)^{\frac{N+2s}{2}}}d\tilde r\\
 &=\frac{(r+R)^{\frac{N-1}{2}}}{R^{N-1}}\int_{2}^{R}(R-\tilde r)^{\frac{N-1}{2}}\frac{1}{(r+\tilde r)^{1+2s}}\int_{\frac{\sqrt{(r+R)(R-\tilde r)}}{(r+\tilde r)}(\mathbb{S}^{N-1}-e_1)}\frac{d\theta}{(1+|\theta|^2)^{\frac{N+2s}{2}}}d\tilde r\\
 &=\frac{(r+R)^{\frac{N-1}{2}}}{R^{N-1}}\int_{0}^{\frac{(R+r)(R-2)}{r+2}}(R+r)^{-2s}(1+\frac{\ov r}{r+R})^{2s-1}\ov r^{\frac{N-1}{2}}(1+\frac{\ov r}{r+R})^{-\frac{N-1}{2}} f\big(\frac{\ov r}{r+R}\big)\frac{d\ov r}{r+R}\\
 &=\frac{(r+R)^{N-1}}{R^{N-1}}(R+r)^{-2s}\int_{0}^{\frac{R-2}{r+2}}(1+\rho)^{2s-1}\Big(\frac{\rho}{1+\rho}\Big)^{\frac{N-1}{2}}f(\rho)d\rho\\
 &\leq \frac{(r+R)^{N-1}}{R^{N-1}}(R+r)^{-2s}\int_{0}^{\frac{R}{2}}(1+\rho)^{2s-1}\int_{\sqrt{\rho(\rho+1)}(\mathbb{S}^{N-1}-e_1)}\frac{d\theta}{(1+|\theta|^2)^{\frac{N+2s}{2}}}d\rho.
\end{align*}
Now since 
\[
 f(\rho)=\int_{\sqrt{\rho(\rho+1)}(\mathbb{S}^{N-1}-e_1)}\frac{d\theta}{(1+|\theta|^2)^{\frac{N+2s}{2}}}=\frac{\pi^{\frac{N-1}{2}}}{\Gamma(\frac{N-1}{2})}\Psi(\frac{1}{2\sqrt{\rho(\rho+1)}})\rightarrow \frac{\pi^{\frac{N-1}{2}}\G(\frac{1+2s}{2})}{\G(\frac{N+2s}{2})}\quad\text{as}\quad\rho\to\infty,
\] 
there exists $m>0$ large enough so that $f(\rho)\leq C(N,s,m)$ for $\rho\geq m$. From this, the continuity of $\rho\mapsto f(\rho)$, and the estimate above we get 
\begin{align}\label{eq-uniform-estimate-killing-term-1}
 &\int_{-R}^{-2}\cdots d\tilde r\leq\frac{(r+R)^{N-1}}{R^{N-1}}(R+r)^{-2s}\int_{0}^{\frac{R}{2}}(1+\rho)^{2s-1}\int_{\sqrt{\rho(\rho+1)}(\mathbb{S}^{N-1}-e_1)}\frac{d\theta}{(1+|\theta|^2)^{\frac{N+2s}{2}}}d\rho\nonumber\\
&\leq 2^{N-1}R^{-2s}\int_{0}^m(1+\rho)^{2s-1}\ov f(\rho)d\rho+C(N,s,m)\int_{m}^{\frac{R}{2}}(1+\rho)^{2s-1}d\rho\nonumber\\
&\leq C(N,s,m)R^{-2s}(1+R^{2s})\leq C(N,s).
\end{align}
Similarly, by the change of variable $\ov r=\frac{(r+R)(\tilde r+R)}{\tilde r-r}$ we get
\begin{align*}
&\int_{2}^\infty\cdots d\tilde r=\int_{2}^\infty\frac{K_R(r,\tilde r)}{(\tilde r-r)^{1+2s}}d\tilde r\\
&=\frac{(r+R)^{\frac{N-1}{2}}}{R^{N-1}}\int_{2}^\infty (\tilde r+R)^{\frac{N-1}{2}}\frac{1}{(\tilde r-r)^{1+2s}}\int_{\frac{\sqrt{(r+R)(\tilde r+R)}}{(\tilde r-r)}(\mathbb{S}^{N-1}-e_1)}\frac{d\theta}{(1+|\theta|^2)^{\frac{N+2s}{2}}}\\
&=\frac{(r+R)^{\frac{N-1}{2}}}{R^{N-1}}\int_{r+R}^{\frac{(r+R)(R+2)}{2-r}}(R+r)^{-2s}(\frac{\ov r}{r+R}-1)^{2s-1}\ov r^{\frac{N-1}{2}}(\frac{\ov r}{r+R}-1)^{-\frac{N-1}{2}} g\big(\frac{\ov r}{r+R}\big)\frac{d\ov r}{r+R}\\
&=\frac{(r+R)^{N-1}}{R^{N-1}}(R+r)^{-2s}\int_{1}^{\frac{R+2}{2-r}}(\rho-1)^{2s-1}\rho^{\frac{N-1}{2}}(\rho-1)^{-\frac{N-1}{2}} g(\rho)d\rho\\
&\leq 2^{N-1}R^{-2s}\int_{1}^{R+2}(\rho-1)^{2s-1}\big(\frac{\rho}{\rho-1}\big)^{\frac{N-1}{2}} g(\rho)d\rho\\
&\leq 2^{N-1}R^{-2s}\int_{0}^{2R}\rho^{2s-1}\big(\frac{\rho+1}{\rho}\big)^{\frac{N-1}{2}} g(\rho+1)d\rho=2^{N-1}\int_{0}^2\rho^{2s-1}\Big[\frac{\rho+1/R}{\rho}\Big]^{\frac{N-1}{2}}g(R\rho+1)d\rho,
\end{align*}
with 
\be\label{eq-def-ov-g}
 g(\rho)=\int_{\sqrt{\rho(\rho-1)}(\mathbb{S}^{N-1}-e_1)}\frac{d\theta}{(1+|\theta|^2)^{\frac{N+2s}{2}}}\quad\text{for all}\quad\rho>1.
\ee
Since \[g(R\rho+1)=\int_{\sqrt{R\rho(R\rho+1)}(\mathbb{S}^{N-1}-e_1)}\frac{d\theta}{(1+|\theta|^2)^{\frac{N+2s}{2}}}\rightarrow \frac{\pi^{\frac{N-1}{2}}\G(\frac{1+2s}{2})}{\G(\frac{N+2s}{2})}\quad\text{uniformly in}\;\; \rho,\] and that $\frac{\rho+1/R}{\rho}$ remain bounded for $R$ large enough, we deduce from above that 
\be\label{eq-uniform-estimate-killing-term-2}
\int_{2}^\infty\frac{K_R(r,\tilde r)}{(\tilde r-r)^{1+2s}}d\tilde r\leq C(N,s)\int_{0}^2\rho^{2s-1}d\rho\leq C(N,s)\qquad\text{as}\ R\to\infty.
\ee
Estimates \eqref{eq-uniform-estimate-killing-term-1} and \eqref{eq-uniform-estimate-killing-term-2} yield
\be\label{eq-uniform-estimate-killing-term}
\int_{(-R,\infty)\setminus(-2,2)}\frac{K_R(r,\tilde r)}{|r-\tilde r|^{1+2s}}d\tilde r\leq C(N,s),\qquad\text{for $R$ sufficiently large.}
\ee
Hence,
\begin{align}\label{eq-uniform-estimate-killing-form}
\int_{0}^1\phi_1^2(r)\int_{(-R,\infty)\setminus(-2,2)}\frac{K_R(r,\tilde r)}{|r-\tilde r|^{1+2s}}d\tilde r\leq C(N,s)\int_{0}^1\phi_1^2(r)dr.
\end{align}
Putting together \eqref{eq-norm-of-rescaled-eigenfunctions}, \eqref{eq-x-1.8}, and \eqref{eq-uniform-estimate-killing-form} we obtain that 
\begin{align}\label{eq-uniform-rescaled-first-eigenfunction}
&R^{-(N-1)}\int_{\R^{2N}}\frac{(\Phi_R(x)-\Phi_R(y))^2}{|x-y|^{N+2s}}dxdy\leq C(N,s)\qquad\text{as}\ R\to\infty.
\end{align} 
Next, since 
\be\label{eq-L^2-limit-of-rescaled-eigenfunction}
R^{-(N-1)}\int_{A_R}\Phi_R^2(x)dx=\o_N\int_{0}^1(1+r/R)^{N-1}\phi_1^2(r)dr\rightarrow \o_N\int_{0}^1\phi_1^2dr=\o_N
\ee
as $R\to\infty$, it follows that
\begin{align*}
\l_{1,s}(A_R) &= \inf_{u\in\cH_0^s(A_R)}\left\lbrace \frac{\frac{b_{N,s}}{2}\int_{\R^{2N}}\frac{(u(x)-u(y))^2}{|x-y|^{N+2s}}dxdy}{\int_{A_R}u^2(x)dx}\right\rbrace \nonumber\\
&\leq \frac{\frac{b_{N,s}}{2}\int_{\R^{2N}}\frac{(\Phi_R(x)-\Phi_R(y))^2}{|x-y|^{N+2s}}dxdy}{\int_{A_R}\Phi_R^2(x)dx}\leq C(N,s)
\end{align*}
as $R\to\infty$ by \eqref{eq-uniform-rescaled-first-eigenfunction} and \eqref{eq-L^2-limit-of-rescaled-eigenfunction}. This shows the claim holds for $R\geq R_0$ for some fixed $R_0\geq 1$. The assertion of the Lemma then easily follows.
\end{proof}
Next, we prove 
\begin{lemma}\label{lem-3.2}
    Let $\phi_2$ be a second fractional eigenfunction of the interval $(0,1)$ and define   
\be\label{eq-def-alpha-R}
\alpha_R=\int_{A_R}\phi_{1,A_R}(x)\phi_2(|x|-R)\,dx,
\ee
where $\phi_{1,A_R}$ is the first normalized eigenfunction of $A_R$. Then, we have
\be\label{eq-limit-of-alpha-R}
R^{-\frac{N-1}{2}}\alpha_R\to 0\quad \text{as $R\to\infty$}.
\ee
\end{lemma}
\begin{proof}
Let $\phi_{1,A_R}$ be the first normalized eigenfunction so that 
\begin{align}\label{eq-normalization-condition}
1=\int_{A_R}\phi_{1,A_R}^2(x)dx&=\o_N \int_{0}^1(1+\frac{r}{R})^{N-1}R^{N-1}\phi^2_{1,A_R}(r+R)dr\\
&\geq\o_N\int_{0}^1\Big[R^{\frac{N-1}{2}}\phi_{1,A_R}(\cdot+R)\Big]^2dr.\nonumber
\end{align}
Taking into account Lemma \ref{eq-uniform-estimate-lambda-1-A-R} and   passing into polar coordinates in the identity
$$
\frac{2}{b_{N,s}}\l_{1,s}(A_R)\int_{A_R}\phi_{1,A_R}^2(x)dx=\int_{\R^{2N}}\frac{\big(\phi_{1,A_R}(x)-\phi_{1,A_R}(y)\big)^2}{|x-y|^{N+2s}}dxdy,
$$
we obtain for $R$ sufficiently large 
\begin{align}
&\ov C(N,s)\geq\frac{2}{b_{N,s}}\l_{1,s}(A_R)\int_{0}^1(1+r/R)^{N-1}\Big[R^{\frac{N-1}{2}}\phi_{1,A_R}(\cdot+R)\Big]^2dr\nonumber\\
&=\int_{-R}^\infty\int_{-R}^\infty\frac{\Big(R^{\frac{N-1}{2}}\phi_{1,A_R}(r+R)-R^{\frac{N-1}{2}}\phi_{1,A_R}(\tilde r+R)\Big)^2}{|r-\tilde r|^{1+2s}}K_R(\tilde r,r)drd\tilde r\nonumber\\
&=\int_{-a}^a \int_{-a}^a\frac{\big(R^{\frac{N-1}{2}}\phi_{1,A_R}(r+R)-R^{\frac{N-1}{2}}\phi_{1,A_R}(\tilde r+R)\big)^2}{|r-\tilde r|^{1+2s}}K_R(r,\tilde r)drd\tilde r\nonumber\\
&\qquad +2\int_{0}^1\Big[R^{\frac{N-1}{2}}\phi_{1,A_R}(r)\Big]^2\int_{(-R,\infty)\setminus(-a,a)}\frac{K_R(r,\tilde r)}{|r-\tilde r|^{1+2s}}d\tilde r\nonumber\\
&\geq\int_{-a}^a \int_{-a}^a\frac{\big(R^{\frac{N-1}{2}}\phi_{1,A_R}(r+R)-R^{\frac{N-1}{2}}\phi_{1,A_R}(\tilde r+R)\big)^2}{|r-\tilde r|^{1+2s}}K_R(r,\tilde r)drd\tilde r\nonumber\\
&\geq C(N,s)\Big[R^{\frac{N-1}{2}}\phi_{1,A_R}(\cdot+R)\Big]^2_{H^s(-a,a)}\label{eq-loc-energy-estimate-of-the-rescaled-first-eigenfunction}
\end{align}
for all $a\in(1,\infty)$, where $K_R(\cdot,\cdot)$ is given by \eqref{eq-def-K-R}. In the last line we used that for $R$ sufficiently large we have

\be\label{eq-uniform-estimate-of-K-R}
c(N,s)\geq K_R(r,\tilde r)\geq C(N,s)\qquad\forall\,\,r,\tilde r\in [-a,a],\ \text{for any $a>0$.}
\ee

Hence, we may assume, up to passing to a subsequence still denoted by $R$, that:
\be\label{eq-weak-limit-of-the-rescaled-eigenfunction}
R^{\frac{N-1}{2}}\phi_{1,A_R}(\cdot+R)\rightarrow u_\infty\quad\text{weakly in $H^s_{loc}(\R)$},\qquad
R^{\frac{N-1}{2}}\phi_{1,A_R}(\cdot+R)\rightarrow u_\infty\quad\text{in $L^2(0,1)$}
\ee
and
\be\label{eq-pointwise-limit-res-eigen}
R^{\frac{N-1}{2}}\phi_{1,A_R}(\cdot+R)\rightarrow u_\infty\qquad\text{pointwisely in $(0,1)$.}
\ee
Note that the limiting function $u_{\infty}$ is nontrivial. Indeed, we have
$$
\int_0^1 u_{\infty}^2(r)\ dr=\lim_{R\to\infty}\int_0^1 \Big[R^{\frac{N-1}{2}}\phi_{1,A_r}(r+R)\Big]^2\ dr=\lim_{R\to\infty}\int_{0}^1(1+\frac{r}{R})^{N-1}R^{N-1}\phi^2_{1,A_R}(r+R)dr=\frac{1}{\omega_N}
$$
To see the equation that $u_\infty$ solves, we let $\psi\in C^\infty_c(0,1)$ and define $\psi_R(x) =R^{-\frac{N-1}{2}}\psi(|x|-R)$ so that $\psi_R\in C^\infty_c(A_R)$. Then, we get as above 
\begin{align}
&\frac{2}{b_{N,s}}\l_{1,s}(A_R)\int_{0}^1(1+\frac{r}{R})^{N-1}R^{\frac{N-1}{2}}\phi_{1,A_R}(r+R)\psi(r)dr=\l_{1,s}(A_R)\frac{1}{\o_N}\int_{A_R}\phi_{1,A_R}(x)\psi_R(x)dx\nonumber\\
&=\int_{-2}^2 \int_{-2}^2 \frac{\big(R^{\frac{N-1}{2}}\phi_{1,A_R}(r+r)-R^{\frac{N-1}{2}}\phi_{1,A_R}(\tilde r+R)\big)\big(\psi(r)-\psi(\tilde r)\big)}{|r-\tilde r|^{1+2s}}K_R(r,\tilde r)drd\tilde r\nonumber\\
&+2\int_{0}^1R^{\frac{N-1}{2}}\phi_{1,A_R}(r+R)\psi(r)\int_{(-R,\infty)\setminus(-2,2)}\frac{K_R(r,\tilde r)}{|r-\tilde r|^{1+2s}}d\tilde rdr.\label{eq-splitting-quadratic-form}
\end{align}
By the strong convergence there exists a subsequence, still denoted by $R$, so that 
\be\label{eq-uniform-estimate-res-phi-A_R-2}
|R^{\frac{N-1}{2}}\phi_{1,A_R}(\cdot+R)|\leq h\quad\text{with} \quad h\in L^2(0,1).
\ee
Moreover, for any $r\in(0,1)$, we know from above that 
\begin{align}
\int_{(-R,\infty)\setminus (-2,2)}\frac{K_R(r,\tilde r)}{|r-\tilde r|^{1+2s}}d\tilde r &= \int_{-R}^{-2}\frac{K_R(r,\tilde r)}{|r-\tilde r|^{1+2s}}d\tilde r+\int_{2}^\infty\frac{K_R(r,\tilde r)}{|r-\tilde r|^{1+2s}}d\tilde r\nonumber\\
&=\frac{(r+R)^{N-1}}{R^{N-1}}(R+r)^{-2s}\int_{0}^{\frac{R-2}{r+2}}(1+\rho)^{2s-1}\Big(\frac{\rho}{1+\rho}\Big)^{\frac{N-1}{2}} f(\rho)d\rho\nonumber\\
&\qquad+\frac{(r+R)^{N-1}}{R^{N-1}}(R+r)^{-2s}\int_{1}^{\frac{R+2}{2-r}}(\rho-1)^{2s-1}\rho^{\frac{N-1}{2}}(\rho-1)^{-\frac{N-1}{2}} g(\rho)d\rho\nonumber\\
&\leq C(N,s),\label{eq-uniform-estimate-Killing-form-which-resp-K-R}
\end{align}
where $ f$ and $ g$ are defined in \eqref{eq-def-ov-f} and \eqref{eq-def-ov-g} respectively. Putting together \eqref{eq-uniform-estimate-res-phi-A_R-2} and \eqref{eq-uniform-estimate-Killing-form-which-resp-K-R} gives
\be\label{eq-eq-1}
\Big|R^{\frac{N-1}{2}}\phi_{1,A_R}(r+R)\psi(r)\int_{\R\setminus (0,1)}\frac{K_R(r,\tilde r)}{|r-\tilde r|^{1+2s}}d\tilde r\Big|\leq C(N,s)h(r)\psi(r).
\ee
Moreover,
\be\label{eq-eq-2}
\int_{0}^1 h(r)\psi(r)dr<\infty.
\ee
Consequently, the dominated convergence theorem yields
\begin{align}
&\lim_{R\to\infty}\int_{0}^1R^{\frac{N-1}{2}}\phi_{1,A_R}(r+R)\psi(r)\int_{\R\setminus (-2,2)}\frac{K_R(r,\tilde r)}{|r-\tilde r|^{1+2s}}d\tilde rdr\nonumber\\
&=\int_{0}^1u_\infty(r)\psi(r)\lim_{R\to\infty}\int_{\R\setminus (0,1)}\frac{K_R(r,\tilde r)}{|r-\tilde r|^{1+2s}}d\tilde rdr\nonumber\\
&=\int_{0}^1u_\infty(r)\psi(r)\Big[\lim_{R\to\infty}\big(\frac{r+R}{R}\big)^{N-1}\big(\frac{R-2}{R+r}\big)^{2s}\int_{0}^{\frac{1}{r+2}}(\frac{1}{R-2}+\rho)^{2s-1}\Big(\frac{(R-2)\rho}{1+(R-2)\rho}\Big)^{\frac{N-1}{2}} f((R-2)\rho)d\rho\Big]dr\nonumber\\
&+\int_{0}^1u_\infty(r)\psi(r)\Big[\lim_{R\to\infty}\big(\frac{r+R}{R}\big)^{N-1}\big(\frac{R+2}{R+r}\big)^{2s}\int_{\frac{1}{R+2}}^{\frac{1}{2-r}}\big(\rho-\frac{1}{R+2}\big)^{2s-1}\Big(\frac{(R+2)\rho}{(R+2)\rho-1}\Big)^{\frac{N-1}{2}} g((R+2)\rho)d\rho\Big]dr\nonumber\\
&=\k_{N,s}\int_{0}^1u_\infty(r)\psi(r)\Big(\int_{0}^{\frac{1}{r+2}}\rho^{2s-1}d\rho+\int_{0}^{\frac{1}{2-r}}\rho^{2s-1}d\rho\Big)dr\nonumber\\
&=\k_{N,s}\int_{0}^1u_\infty(r)\psi(r)\frac{1}{2s}\Big[(r+2)^{-2s}+(2-r)^{-2s}\Big]dr\nonumber\\
&=\k_{N,s}\int_{0}^1u_\infty(r)\psi(r)\Big(\int_{-\infty}^{-2}\frac{d\tilde r}{|r-\tilde r|^{1+2s}}+\int_{2}^\infty\frac{d\tilde r}{|r-\tilde r|^{1+2s}}\Big)\nonumber\\
&=\k_{N,s}\int_{0}^1u_\infty(r)\psi(r)\int_{\R\setminus(-2,2)}\frac{d\tilde r}{|r-\tilde r|^{1+2s}}\label{eq-limit-of-modified-Killing-term}
\end{align}
with 
\begin{align*}
\kappa(N,s)=\lim_{R\to\infty} f((R-2)\rho)=\lim_{R\to\infty} g((R+2)\rho)=\frac{\pi^{\frac{N-1}{2}}\G(\frac{1+2s}{2})}{\G(\frac{N+2s}{2})}.
\end{align*}
As for the first integral in \eqref{eq-splitting-quadratic-form}, we write 
\begin{align}
&\int_{-2}^2\int_{-2}^2\frac{\big(R^{\frac{N-1}{2}}\phi_{1,A_R}(r+R)-R^{\frac{N-1}{2}}\phi_{1,A_R}(\tilde r+R)\big)\big(\psi(r)-\psi(\tilde r)\big)}{|r-\tilde r|^{1+2s}}K_R(r,\tilde r)drd\tilde r\nonumber\\
&=\k(N,s)\int_{-2}^2\int_{-2}^2\frac{\big(R^{\frac{N-1}{2}}\phi_{1,A_R}(r+R)-R^{\frac{N-1}{2}}\phi_{1,A_R}(\tilde r+R)\big)\big(\psi(r)-\psi(\tilde r)\big)}{|r-\tilde r|^{1+2s}}drd\tilde r\nonumber\\
&\quad+\int_{-2}^2\int_{-2}^2\frac{\big(R^{\frac{N-1}{2}}\phi_{1,A_R}(r+R)-R^{\frac{N-1}{2}}\phi_{1,A_R}(\tilde r+R)\big)\big(\psi(r)-\psi(\tilde r)\big)}{|r-\tilde r|^{1+2s}}\Big(K_R(r,\tilde r)-\k(N,s)\Big)drd\tilde r\nonumber\\
&=:\calE_{(-2,2)}\big(R^{\frac{N-1}{2}}\phi_{1,A_R},\psi(\cdot+R)\big)+\calE_{(-2,2),R}\big(R^{\frac{N-1}{2}}\phi_{1,A_R}(\cdot+R),\psi\big)\label{eq-1.46}.
\end{align}
On one hand, by the weak convergence \eqref{eq-weak-limit-of-the-rescaled-eigenfunction} we have
\begin{align}
&\lim_{R\to\infty}\calE_{(-2,2)}\big(R^{\frac{N-1}{2}}\phi_{1,A_R}(\cdot+R),\psi\big)=\k(N,s)\int_{-2}^2\int_{-2}^2\frac{\big(u_\infty(r)-u_\infty(\tilde r)\big)\big(\psi(r)-\psi(\tilde r)\big)}{|r-\tilde r|^{1+2s}}drd\tilde r.\label{eq-1.47}
\end{align}
On the other hand, 
\begin{align}
&\Big|\calE_{(-2,2),R}\big(R^{\frac{N-1}{2}}\phi_{1,A_R}(\cdot+R),\psi\big)\Big|\nonumber\\
&\leq  \sup_{r,\tilde r\in (-2,2)}\big|K_R(\cdot,\cdot)-\k(N,s)\big|\int_{-2}^2\int_{-2}^2\frac{\big|\big(R^{\frac{N-1}{2}}\phi_{1,A_R}(r+R)-R^{\frac{N-1}{2}}\phi_{1,A_R}(\tilde r+R)\big)\big(\psi(r)-\psi(\tilde r)\big)\big|}{|r-\tilde r|^{1+2s}}\nonumber\\
&\leq \sup_{r,\tilde r\in (-2,2)}\big|K_R(\cdot,\cdot)-\k(N,s)\big|[R^{\frac{N-1}{2}}\phi_{1,A_R}(\cdot+R)]_{H^s(-2,+2)}[\psi]_{H^s(-2,+2)}\nonumber\\
&\leq C(N,s,\psi)\sup_{r,\tilde r\in (-2,2)}\big|K_R(\cdot,\cdot)-\k(N,s)\big|\rightarrow 0\quad\text{as}\quad R\to\infty\label{eq-1.48},
\end{align}
where we used \eqref{eq-uniform-convergence-K-R} and \eqref{eq-loc-energy-estimate-of-the-rescaled-first-eigenfunction}. Plugging \eqref{eq-1.47} and \eqref{eq-1.48} into \eqref{eq-1.46} yields 
\begin{align}
&\int_{-2}^2\int_{-2}^2\frac{\big(R^{\frac{N-1}{2}}\phi_{1,A_R}(r+R)-R^{\frac{N-1}{2}}\phi_{1,A_R}(\tilde r+R)\big)\big(\psi(r)-\psi(\tilde r)\big)}{|r-\tilde r|^{1+2s}}K_R(r,\tilde r)drd\tilde r\nonumber\\
&=\k(N,s)\int_{-2}^2\int_{-2}^2\frac{\big(u_\infty(r)-u_\infty(\tilde r)\big)\big(\psi(r)-\psi(\tilde r)\big)}{|r-\tilde r|^{1+2s}}drd\tilde r.\label{eq-limit-of-restricted-bilinear-form}
\end{align}
Now with \eqref{eq-limit-of-modified-Killing-term} and \eqref{eq-limit-of-restricted-bilinear-form} we can pass to a limit in \eqref{eq-splitting-quadratic-form} to obtain (along a subsequence)
\begin{align*}
\l_\infty\int_{0}^1u_\infty(r)\psi(r)dr&=\kappa(N,s)\frac{c_{N,s}}{2}\int_{\R^2}\frac{(u_\infty(r)-u_\infty(\tilde r))(\psi(r)-\psi(\tilde r))}{|r-\tilde r|^{1+2s}}drd\tilde r \\
&=\frac{b_{1,s}}{2}\int_{\R^2}\frac{(u_\infty(r)-u_\infty(\tilde r))(\psi(r)-\psi(\tilde r))}{|r-\tilde r|^{1+2s}}drd\tilde r\qquad\text{for all}\,\psi\in C^\infty_c(0,1),
\end{align*}
where we used that $b_{N,s}\k(N,s)=\frac{s4^s\G(\frac{N}{2}+s)}{\pi^{\frac{N}{2}}\G(1-s)}\frac{\pi^{\frac{N-1}{2}}\G(s+\frac{1}{2})}{\G(\frac{N}{2}+s)}=\frac{s4^s\G(s+\frac{1}{2})}{\sqrt{\pi}\G(1-s)}=b_{1,s}$. Since $u_\infty\geq 0$ in $(0,1)$ by the pointwise convergence \eqref{eq-pointwise-limit-res-eigen}, it follows that $\l_\infty=\l_{1,s}(0,1)$ and $u_\infty=\phi_1>0$ by the normalization condition. Since  $\lambda_{1,s}(0,1)$ is simple, we deduce that the whole sequence $(R^{\frac{N-1}{2}}\phi_{1,A_R}(\cdot+R))_R$ converges to $\phi_1$ in $L^2((0,1))$. In conclusion, we have for $R\to\infty$
\begin{align*}
 \frac{\al_R}{R^{\frac{N-1}{2}}\o_N} &=\int_{0}^1(1+\frac{r}{R})^{N-1}R^{\frac{N-1}{2}}\phi_{1,A_R}(r+R)\phi_2(r)dr\rightarrow \int_{0}^1u_\infty(r)\phi_2(r)dr=\int_{0}^1\phi_1(r)\phi_2(r)dr=0.
\end{align*}
Hence, \eqref{eq-limit-of-alpha-R} holds.
\end{proof}
From Lemma \ref{eq-uniform-estimate-lambda-1-A-R}, and Lemma \ref{lem-3.2} we deduce 
\begin{lemma}\label{lem-3.3}
There exists $C(N,s)>0$ such that 
\be
\l_{2,s}(A_R)\leq C(N,s)\quad\text{for all $R\geq 1$.}
\ee
\end{lemma}
\begin{proof}
As in the proof of Lemma \ref{eq-uniform-estimate-lambda-1-A-R} it is enough to show that the claim holds for $R\geq R_0$ for some $R_0\geq 1$. Let $\phi_2$ be a (normalized) second eigenfunction of the interval $(0,1)$. We define 
$$
\Psi_R(x) = \phi_2(|x|-R)\quad\,\text{for all}\ x\,\in\R^N.
$$
Next, we let $\Upsilon_R$ be the projection of $\Psi_R$ into the subspace $\R\phi_{1,A_R}$, where $\phi_{1,A_R}$ is the first eigenfunction of $A_R$. That is,
$$
\Upsilon_R:=\Psi_R-\Big[\int_{A_R}\phi_{1,A_R}(x)\Psi_R(x)\,dx\Big]\phi_{1,A_R}=\Psi_R-\al_R\phi_{1,A_R},
$$
where $\alpha_R$ is defined as in Lemma \ref{lem-3.2}.
It is clear that $\Upsilon_R=0$ in $\R^N\setminus A_R$ and that $\int_{A_R}\Upsilon_R(x)\phi_{1,A_R}(x)dx=0$. Moreover, $\Upsilon_R\in H^s(\R^N)$ by construction. By the variational characterization of $\l_{2,s}(A_R)$, we get 
\begin{align}
\frac{2}{b_{N,s}}&\l_{2,s}(A_R)\leq \frac{\int_{\R^{2N}}\frac{\big(\Upsilon_R(x)-\Upsilon_R(y)\big)^2}{|x-y|^{N+2s}}dxdy}{\int_{A_R}\Upsilon_R^2(x)dx}\nonumber\\
&=\frac{\int_{\R^{2N}}\frac{\big(\Psi_R(x)-\Psi_R(y)\big)^2}{|x-y|^{N+2s}}dxdy+\al_R^2\int_{\R^{2N}}\frac{(\phi_{1,A_R}(x)-\phi_{1,A_R}(y))^2}{|x-y|^{N+2s}}dxdy-2\al_R^2\l_{1,s}(A_R)}{\int_{A_R}\Psi_R^2(x)dx-\al_R^2}\nonumber\\
&=\frac{\int_{\R^{2N}}\frac{\big(\Psi_R(x)-\Psi_R(y)\big)^2}{|x-y|^{N+2s}}dxdy-\al_R^2\l_{1,s}(A_R)}{\int_{A_R}\Psi_R^2(x)\,dx-\al_R^2}.
\label{eq-estimate-lambda-2-A-R}
\end{align}
First, we note that the estimate \eqref{eq-uniform-rescaled-first-eigenfunction} does not use the fact that $\phi_1$ is a first eigenfunction. It only uses that $\phi_1\in \cH^s_0((0,1))$. In particular, we also have
\be\label{eq-uniform-estimate-res-second-eigenfunction}
 R^{-(N-1)}\int_{\R^{2N}}\frac{\big(\Psi_R(x)-\Psi_R(y)\big)^2}{|x-y|^{N+2s}}dxdy\leq C(N,s)\quad\text{as $R\to\infty$}.
\ee
Moreover, since $\lambda_{1,2}(A_R)$ is uniformly bounded, we have, using Lemma \ref{lem-3.3} that 
\be\label{eq-37}
\lim_{R\to\infty}R^{-(N-1)}\alpha_R^2\lambda_{1,s}(A_R)=0=\lim_{R\to\infty}R^{-(N-1)}\alpha_R^2.
\ee
It is also clear that 
\be\label{eq-38}
R^{-(N-1)}\int_{A_R}\Psi_R^2(x)\,dx=\o_N\int_{0}^1\phi_2^2(r)(1+r/R)^{N-1}dr\to 1\quad\text{as $R\to\infty$}.
\ee
We deduce from \eqref{eq-uniform-estimate-res-second-eigenfunction}, \eqref{eq-37}, and \eqref{eq-38} that the RHS of \eqref{eq-estimate-lambda-2-A-R}  is uniformly bounded for $R$ sufficiently large.
\end{proof}

\section{Uniform estimate of fractional normal derivative of second radial eigenfunctions of \texorpdfstring{$A_R$}{AR}}\label{sec-4}
For simplicity we  denote by $u_R$ a fixed normalized radial eigenfunction (if any) corresponding to $\l_{2,s}(A_R)$ and define $w_R:\R\to \R$ by $w_R(r):=R^{\frac{N-1}{2}}\ov {u}_R(r+R)$ for $r>0$, where $u_R(x)=\ov{u}_R(|x|)$, and $w_R(r)=0$ for $r\leq 0$. Consider the functions $\ov K_R$ and $V_R$ defined respectively by
$$
\ov K_R(r,\tilde r)=\frac{1}{|r-\tilde r|^{1+2s}}\frac{(r+R)^{\frac{N-1}{2}}(\tilde r+R)^{\frac{N-1}{2}}}{R^{N-1}}\int_{\frac{\sqrt{(r+R)(\tilde r+R)}}{|r-\tilde r|}(\mathbb{S}^{N-1}-e_1)}\frac{d\theta}{(1+|\theta|^2)^{\frac{N+2s}{2}}}\quad \text{for all $r\neq \tilde r$},
$$
and 
$$
V_R(r)= 2\int_{(-R,\infty)\setminus(-2,2)}\ov K_R(r,\tilde r)d\tilde r.
$$
Note that from \eqref{eq-uniform-estimate-killing-term} and \eqref{eq-uniform-estimate-of-K-R} we have, for $R$ sufficiently large
\begin{align}
\ov C(N,s)|r-\ti r|^{-1-2s}\leq \ov K_R(r,\tilde r)&\leq C(N,s)|r-\ti r|^{-1-2s}\quad\text{for all $r,\tilde r\in [-a,a]$, $a>0$ and}
\label{eq-uniform-ellipticity}\\
0< V_R(r)&\leq C(N,s)\label{eq-uniform-estimate-potential}.
\end{align}
We need the following in the proof of Proposition \ref{eq-uniform-estimate-fract-normal-der-of-w-R} below.

\begin{lemma}\label{sec4:preparations}
Let $w_R$ be defined as above. Then, $w_R\in H^s(\R)$ with $\frac{1}{\omega_N}(\frac{R}{1+R})^{N-1}\leq \|w_R\|_{L^2(\R)}\leq \frac{1}{\omega_N}$,
\begin{equation}\label{sec4:preparations-eq1}
\quad \|w_R\|_{H^s(\R)}\leq c(N,s),\quad \text{and}\quad \|w_R\|_{L^{\infty}(\R)}\leq c(N,s)\quad\text{for all $R>1$.}
\end{equation}
Moreover, it solves weakly the equation
\be\label{eq-equation-rescaled-second-eigenfunction}
\calL_{K_R}w_R+w_R V_R=\l_{2,s}(A_R)\big(1+r/R\big)^{N-1} w_R\qquad\text{in}\qquad (0,1),
\ee
where
$$
\langle\calL_{K_R}u,w\rangle:=\frac{b_{N,s}}{2}\int_{\R^2}(u(r)-u(\ti r))(w(r)-w(\ti r))K_R(r,\ti r)drd\ti r
$$
with $$K_R(r,\ti r)=1_{(-2,2)}(r)1_{(-2,2)}(\tilde r)\ov{K}_R(r,\tilde r).$$
\end{lemma}
\begin{proof}
	First note that we have
	$$
	\|w_R\|_{L^2(\R)}^2= \int_0^1 R^{N-1} \overline{u}^2(r+R)\ dr=\int_R^{R+1}R^{N-1} \bar{u}^2(r)\ dr=\int_{A_R}\frac{R^{N-1}}{|x|^{N-1}\omega_N}u_R^2(x)\ dx.
    $$
    This shows the first claim by the $L^2$-normalization of $u_R$. The fact that $w_R\in H^s_{loc}(\R^N)$ follows as in the proof of Lemma \ref{lem-3.2}. Indeed, similarly to the calculation in equation \eqref{eq-loc-energy-estimate-of-the-rescaled-first-eigenfunction} we have the estimate 
\begin{align*}
\frac{2}{b_{N,s}}\l_{2,s}(A_R)&=\int_{\R^{2N}}\frac{\big(u_{R}(x)-u_{R}(y)\big)^2}{|x-y|^{N+2s}}dxdy\\
&\geq\int_{-a}^a \int_{-a}^a \big(R^{\frac{N-1}{2}}\ov u_R(r+R)-R^{\frac{N-1}{2}}\ov u_R(\tilde r+R)\big)^2 \ov K_R(r,\tilde r)drd\tilde r\nonumber\\
&\geq \bar{C}(N,s)\big[R^{\frac{N-1}{2}}\ov u_R(\cdot+R)\big]^2_{H^s(-a,a)}\quad\text{for all $a>0$}.
\end{align*}
In particular, $[w_R]^2_{H^s((-2,2))}$ is bounded independently of $R$ by Lemma \ref{lem-3.3}. Now, by definition, $w_R=0$ in $\R\setminus(0,1)$ and thus the quantity
$$
\int_{\R^{2}}\frac{(w_R(r)-w_R(\tilde{r}))^2}{|r-\tilde{r}|^{1+2s}}\ drd\tilde{r}=\int_{(-2,2)^2}\frac{(w_R(r)-w_R(\tilde{r}))^2}{|r-\tilde{r}|^{1+2s}}\ drd\tilde{r}+2\int_0^1w_R^2(r) \int_{\R\setminus(-2,2)}|r-\tilde{r}|^{-1-2s}\ d\tilde{r}dr
$$
is also bounded independently of $R$. Before showing the boundedness, we show next equation \eqref{eq-equation-rescaled-second-eigenfunction}. Indeed, this follows from the identity 
\begin{align*}
&\frac{2}{b_{N,s}}\l_{2,s}(A_R)\int_{0}^1(1+\frac{r}{R})^{N-1}R^{\frac{N-1}{2}}\ov u_{R}(r+R)\psi(r)dr=\l_{2,s}(A_R)\frac{1}{\o_N}\int_{A_R}u_{R}(x)\psi_R(x)dx\nonumber\\
&=\int_{-2}^2 \int_{-2}^2 \big(R^{\frac{N-1}{2}}\ov u_R(r+R)-R^{\frac{N-1}{2}}\ov u_R(\tilde r+R)\big)\big(\psi(r)-\psi(\tilde r)\big)\ov K_R(r,\tilde r)drd\tilde r\nonumber\\
&+2\int_{0}^1 R^{\frac{N-1}{2}}\ov u_R(r+R)\psi(r)\int_{(-R,\infty)\setminus(-2,2)}\ov K_R(r,\tilde r)d\tilde rdr.
\end{align*}
for $\psi\in C^\infty_c(0,1)$ and with $\psi_R(x)=R^{-\frac{N-1}{2}}\psi(|x|-R)$. To see now the boundedness, let $\tilde{K}_R(r,\tilde{r})=K_R(r,\tilde{r})$ if $r,\tilde{r}\in(-2,2)$ and $\tilde{K}_R(r,\tilde{r})=|r-\tilde{r}|^{-1-2s}$ otherwise. Then 
$$
\min\Big(\ov C(N,s),1\Big)|r-\ti r|^{-1-2s}\leq \tilde K_R(r,\tilde r)\leq \max\Big(C(N,s),1\Big)|r-\ti r|^{-1-2s} \quad \text{for all $r,\tilde{r}\in\R$}
$$
and $|w_R|$ satisfies weakly
$$
\calL_{\tilde{K}_R}|w_R| \leq \Big(\l_{2,s}(A_R)\big(1+r/R\big)^{N-1} -V_R+\int_{\R\setminus(-2,2)} |r-\tilde{r}|^{-1-2s}\ d\tilde{r}\Big) |w_R| \leq C'(N,s)|w_R| \qquad\text{in}\qquad (0,1),
$$
where
$$
\langle\calL_{\tilde K_R}u,w\rangle:=\frac{b_{N,s}}{2}\int_{\R^2}(u(r)-u(\ti r))(w(r)-w(\ti r))\tilde K_R(r,\ti r)drd\ti r.
$$
By \cite[Theorem 4.1]{FJ22}, it follows that $|w_R|\in L^{\infty}(\R)$ with a bound independent of $R$ (since the kernel $\tilde{K}_R$ has bounds independent of $R$). Thus \eqref{sec4:preparations-eq1} holds.
\end{proof}

Let $w_R\in H^s(\R)$ be as above and $d(\cdot)=\min(\cdot, 1-\cdot)$ be the distance function to the boundary of $(0,1)$. Then we have 
\begin{proposition}\label{eq-prop-uniform-estimate-fractional-normal-derivative}
There exists a constant $C(N,s)>0$  so that for all $\b\in(0,s)$ we have
\be\label{eq-uniform-estimate-fract-normal-der-of-w-R}
\big\|\frac{w_R}{d^s}\big\|_{C^{s-\b}([0,1])}\leq C(N,s)\quad\text{for all $R\geq 2$},
\ee
where $d$ is given as above.
\end{proposition}
\begin{proof}
Let $R\geq 2$,
$$
\t_R(t,r):=\frac{\sqrt{(t+R)(t\pm r+R)}}{r},
$$
and consider the function 
$\ov\l_{K_R}:\R\times [0,\infty)\times\{\pm 1\}\to\R$ defined by 
\begin{align*}
\ov\l_{K_R}(t,r,\pm1)&= r^{1+2s}\ov K_R(t,t\pm r)\\
&=\frac{(t+R)^{\frac{N-1}{2}}(t\pm r+R)^{\frac{N-1}{2}}}{R^{N-1}}\int_{\t_R(t,r)(\mathbb{S}^{N-1}-e_1)}\frac{d\theta}{(1+|\theta|^2)^{\frac{N+2s}{2}}},
\end{align*}
for all $r\neq 0$ and 
$$
\ov\l_{K_R}(t,0,\pm 1):=\lim_{r\to 0^+}\ov \l_{K_R}(t,r,\pm 1)=(1+t/R)^{N-1}\frac{\pi^{\frac{N-1}{2}}\G(\frac{1+2s}{2})}{\G(\frac{N+2s}{2})}.
$$
By Corollary \ref{eq-Holder-reg-2} we have
\begin{align}\label{eq-Holder-estimate-lambda-K-R}
\|\ov \l_{K_R}\|_{C^{0,s+\d}\big([-2,2]\times[0,2]\times \{\pm 1\}\big)}\leq C_0\quad\text{for $R\geq 2$.}
\end{align}
Quoting \cite[Theorem 1.8 and Remark 2.1]{F2017} and since $w_R$ solves \eqref{eq-equation-rescaled-second-eigenfunction}, we get
\begin{align}\label{eq-uniform-estimate-fract-normal-derivative}
\big\|\frac{w_R}{d^{s}}\big\|_{C^{s-\beta}([0,1])}\leq C\big(\|w_R\|_{L^2(0,1)}+\|w_R\|_{\calL^1_s}+\|\l_{2,s}(A_R)\big(1+r/R\big)^{N-1} w_R\|_{L^\infty(0,1)}\big),
\end{align}
 for all $\beta\in (0,s)$ with $C>0$ independent of $R$ by \eqref{eq-Holder-estimate-lambda-K-R}. Here, $\|w_R\|_{\cL^1_s}:=\int_{\R}\frac{|w_R(r)|}{1+r^{1+2s}}dr$ and we have used \eqref{eq-uniform-ellipticity}, \eqref{eq-uniform-estimate-potential}, and \eqref{eq-Holder-estimate-lambda-K-R}. Combining this with Lemma \ref{sec4:preparations}, we see that the RHS of \eqref{eq-uniform-estimate-fract-normal-derivative} is bounded independently of $R$, as claimed.
\end{proof}
\begin{corollary}\label{eq-convergence-of-normal-derivative}
There exists $\phi_2\in\calH^s_0((0,1))$ solving 
\be \label{eq-second-eigenfct}
\Ds\phi_2=\l_{s,2}(0,1)\phi_2\quad\text{in}\quad (0,1)
\ee
so that, along some subsequence still denoted by $R$, it holds
\be
\frac{w_R}{d^s}\rightarrow \frac{\phi_2}{d^s}\quad\text{in}\quad C^0([0,1]).
\ee
\end{corollary}
\begin{proof} By Proposition \ref{eq-prop-uniform-estimate-fractional-normal-derivative} and the Arzela-Ascoli theorem, there exists $q_\infty\in C^{s-\beta}([0,1])$ with $$\frac{w_R}{d^s}\rightarrow q_\infty\qquad\text{for $R\to\infty$ uniformly in $[0,1]$.}$$ It remains to show that $q_\infty=\frac{\phi}{d^s}$ for some $\phi$ solving \eqref{eq-second-eigenfct}. By a similar argument as for the case of the first eigenvalue, one obtains by Lemma \ref{sec4:preparations} (along some subsequence):
\be\label{eq-weak-and-strong-convergence-of-rescaled-second-eigenfunction}
w_R\rightarrow w_{\infty}\quad\text{weakly in}\quad H^s(\R)\qquad\text{and}\qquad w_R\rightarrow w_{\infty}\quad\text{in}\quad L^2((0,1))
\ee
and
\be\label{eq-pointwise-convergence-of-rescaled-second-eigenfunction}
w_R\rightarrow w_{\infty}\quad\text{pointwisely in}\quad (0,1).
\ee
Note that $w_{\infty}$ is nontrivial by the $L^2$-bounds in Lemma \ref{sec4:preparations}.
Passing to a limit in \eqref{eq-equation-rescaled-second-eigenfunction} as we did in the proof of Lemma \ref{lem-3.2}, we get 
\begin{align*}
\l_\infty\int_{0}^1\psi(r)w_\infty(r)dr&=\frac{\k(N,s)b_{N,s}}{2}\int_{\R^2}\frac{(w_\infty(r)-w_\infty(\tilde r))(\psi(r)-\psi(\tilde r))}{|r-\tilde r|^{1+2s}}drd\tilde r\\
&=\frac{b_{1,s}}{2}\int_{\R^2}\frac{(w_\infty(r)-w_\infty(\tilde r))(\psi(r)-\psi(\tilde r))}{|r-\tilde r|^{1+2s}}drd\tilde r,
\end{align*}
for all $\psi\in C^\infty_c(0,1)$, where $\lambda_\infty$ is the limit of $\lambda_{2,s}(A_R)$ along certain subsequence (see Lemma \ref{lem-3.3}). Next, we claim that $\lambda_{\infty}\leq \lambda_{2,s}(0,1)$. Indeed, since $R^{-(N-1)}\a^2_R\to 0$ as $R\to\infty$ by Lemma \ref{lem-3.2} and $\lambda_{1,s}(A_R)$ is uniformly bounded, we find similarly as \eqref{eq-estimate-lambda-2-A-R} in the proof of Lemma \ref{lem-3.3} with the same notation
\begin{align}
\frac{2}{b_{N,s}}\lambda_{2,s}(A_R) \leq \frac{\int_{\R^{2N}}\frac{\big(\Psi_R(x)-\Psi_R(y)\big)^2}{|x-y|^{N+2s}}dxdy-\al_R^2\l_{1,s}(A_R)}{\int_{A_R}\Psi_R^2(x)\,dx-\al_R^2}.
\end{align}
Passing in this inequality to the limit (along a subsequence that we still denote by $R$), we get 
\begin{align*}
\l_\infty=\lim_{R\to\infty}\l_{2,s}(A_R)&\leq \frac{b_{N,s}}{2}\lim_{R\to \infty} \frac{R^{-(N-1)}\int_{\R^{2N}}\frac{\big(\Psi_R(x)-\Psi_R(y)\big)^2}{|x-y|^{N+2s}}dxdy-R^{-(N-1)}\al_R^2\l_{1,s}(A_R)}{R^{-(N-1)}\int_{A_R}\Psi_R^2(x)\,dx-R^{-(N-1)}\al_R^2}\\
&=\frac{b_{N,s}}{2}\lim_{R\to \infty} \frac{R^{-(N-1)}\int_{\R^{2N}}\frac{\big(\Psi_R(x)-\Psi_R(y)\big)^2}{|x-y|^{N+2s}}dxdy}{R^{-(N-1)}\int_{A_R}\Psi_R^2(x)\,dx}\\
&=\frac{b_{N,s}}{2}\lim_{R\to \infty}\frac{\int_{-2}^2\int_{-2}^2\frac{(\phi_2(r)-\phi_2(\tilde r))^2}{|r-\tilde r|^{1+2s}}K_R(r,\tilde r)drd\tilde r+\int_{0}^1\phi_2^2(r)\int_{(-R,+\infty)\setminus (-2,2)} \frac{K_R(r,\tilde r)}{|r-\tilde r|^{1+2s}} d\tilde r dr }{\int_{0}^1(1+\frac{r}{R})^{N-1}\phi^2_2(r)dr}\\
&=\frac{\k(N,s)b_{N,s}}{2}\int_{\R^2}\frac{(\phi_2(r)-\phi_2(\tilde r))^2}{|r-\tilde r|^{1+2s}}drd\tilde r=\l_{2,s}(0,1),
\end{align*}
with $\kappa(N,s)=\frac{\pi^{\frac{N-1}{2}}\G(\frac{1+2s}{2})}{\G(\frac{N+2s}{2})}$, so that $\kappa(N,s)b_{N,s}=b_{1,s}$. Here we used \eqref{eq-uniform-convergence-K-R} and 
$$
\lim_{R\to\infty}\int_{0}^1\phi_2^2(r)\int_{(-R,+\infty)\setminus (-2,2)} \frac{K_R(r,\tilde r)}{|r-\tilde r|^{1+2s}} d\tilde r dr=\kappa(N,s)\int_{0}^1\phi_2^2(r)\int_{\R\setminus (-2,2)} \frac{d\tilde r}{|r-\tilde r|^{1+2s}} dr
$$
by the dominated convergence theorem (see e.g \eqref{eq-limit-of-modified-Killing-term}). This shows $\l_\infty\leq \l_{2,s}(0,1)$. From here, it follows that $w_{\infty}$ is an eigenfunction in $(0,1)$ with corresponding eigenvalue $\l_{\infty}\in \{\l_{1,s}(0,1),\l_{2,s}(0,1)\}$. Following the proof of Lemma \ref{lem-3.2}, let $\phi_1$ denote the first eigenfunction in $(0,1)$ and then it holds $R^{\frac{N-1}{2}}\phi_{1,A_R}(\cdot+R)\to \phi_1$ in $L^2((0,1))$. Then we have
\begin{align*}
\int_{0}^1 \phi_1(r)w_{\infty}(r)\ dr&=\lim_{R\to\infty} \int_0^1(1+r/R)^{N-1}w_R(r)R^{\frac{N-1}{2}}\phi_{1,A_R}(r+R)\ dr=\lim_{R\to\infty}\int_R^{R+1}r^{N-1}u_R(r)\phi_{1,A_R}(r)\ dr\\
&=\frac{1}{\omega_N}\lim_{R\to\infty} \int_{A_R} u_R(x)\phi_{1,A_R}(x)\ dx=0,
\end{align*}
since $u_R$ is an eigenfunction but not the first and thus orthogonal to $\phi_{1,A_R}$. It follows that $w_{\infty}$ is an eigenfunction orthogonal to $\phi_1$. Consequently, $\l_{\infty}=\l_{s,2}(0,1)$ and $w_{\infty}\in \cH^s_0((0,1))$ solves \eqref{eq-second-eigenfct}. By the pointwise convergence \eqref{eq-pointwise-convergence-of-rescaled-second-eigenfunction} and the regularity we see that $\frac{w_R}{d^s}\rightarrow \frac{w_\infty}{d^s}$ in $C^0_{loc}(0,1)$ and by uniqueness of the limit we deduce that $q_\infty=\frac{w_\infty}{d^s}$.
\end{proof}

\begin{remark}
	Following the proofs of Section \ref{sec-3} and Section \ref{sec-4}, along a similar strategy, one can show that if $u_R$ is any normalized radial eigenfunction with profile $\overline{u}_R$ and corresponding eigenvalue $\lambda_R$, then $\lambda_R$ is uniformly bounded in $R$ and after passing to a subsequence, $u_{\infty}:=\lim\limits_{R\to\infty} R^{\frac{N-1}{2}} \ov u_R(\cdot+R)$ is an eigenfunction in $(0,1)$ with corresponding eigenvalue $\lambda_{\infty}$. Here, the convergence is as for $(w_R)_R$ in the above proof. We conjecture, that if $u_R$ is the $n$-th radial eigenfunction, then $u_{\infty}$ will be an $n$-th eigenfunction in $(0,1)$. In order to show this, the Lemmas \ref{lem-3.2} and \ref{lem-3.3} need to be adjusted suitably in the following direction: One of the main steps is to show that the $n$-th eigenfunction in $(0,1)$ translated to $A_R$ as in the proof of Lemma \ref{lem-3.3} is orthogonal in the limit to all the eigenfunctions in $A_R$ belonging to eigenvalues below the $n$-th radial eigenvalue. However, this step is not clear, since we do not have any information on which eigenvalue in $A_R$ actually has a radial eigenfunction.
\end{remark}

\section{Proof of the nonradiality of second Eigenfunctions}\label{sec-5}
\begin{proof}[Proof of Theorem \ref{eq-main-thm}]Assume there is a second radial eigenfunction $u_R$ corresponding to $\l_{2,s}(A_R)$. To reach a contradiction, we prove that such $u_R$ must have Morse index greater than or equal to $N+1$ and this would contradict the fact that $u_R$ is a second eigenfunction. Let $w_R:=R^{\frac{N-1}{2}}\ov{u}_R(\cdot+R)$, we know by Corollary \ref{eq-convergence-of-normal-derivative} that there exists some subsequence still denoted by $R$ along which we have $\frac{w_R}{d^s}\rightarrow \frac{\phi_2}{d^s}$ in $C^0([0,1])$ where $d(r)=\min(r, 1-r)$ and $\phi_2$ is a second eigenfunction. We also know that $\phi_2$ is antisymmetric and it vanishes only at $r=1/2$ see e.g \cite[Theorem 5.2]{Fetal}. Moreover, $\frac{\phi_2}{d^s}(1)\frac{\phi_2}{d^s}(0)<0$ (the latter can be seen, for instance, by applying the Hopf lemma for entire antisymmetric supersolutions in \cite[Proposition 3.3]{FJ}). Consequently, we have 
$$
\Bigg(\frac{u_R}{d^s} \Bigg|_{\partial_{in} A_{R}}\Bigg)\cdot \Bigg(\frac{u_R}{d^s} \Bigg|_{\partial_{out} A_{R}}\Bigg) <0\qquad\text{for}\quad R>0\quad  \text{sufficiently large}.
$$
Here, $d(x)=\min\big(|x|-R, R+1-|x|\big)$ and $\partial_{in} A_{R}$, $\partial_{out} A_{R}$ denote respectively the inner and the outer boundary of the annulus $A_R$, i.e. $\partial_{in} A_{R}=\partial B_R(0)$ and $\partial_{out} A_{R}=\partial B_{R+1}(0)$. Without loss of generality we may assume 
\be\label{Hopf type lemma for the second radial eigenfunction}
\frac{u_R}{d^s}\Big|_{\partial_{out} A_{R}}>0\qquad\text{and}\qquad \frac{u_R}{d^s}\Big|_{\partial_{in} A_{R}}<0.
\ee
For simplicity we let $\ov\psi_R:=\frac{\ov{u}_R(\cdot+R)}{\min(\cdot,1-\cdot)^s}:[0,1]\to\R$ so that $\ov\psi_R(1)>0$ and $\ov\psi_R(0)<0$. In the spirit of \cite{Fetal}, for any fix direction $j\in\{1,\cdots,N\}$, we define 

\be\label{def-dj}
d^j_R=(v^j_R)^+1_{H_+^j}-(v^j_R)^-1_{H_-^j},
\ee
where $$v^j_R:\R^N\to\R^N,\;\;\;x\mapsto v^j_R(x)=\left\{\begin{aligned} &\frac{\partial u_R}{\partial x_j}(x), && \text{if}\;\; x\in A_R;\\
&0&&\text{if}\;\;x\in\R^N\setminus A_R.\end{aligned}\right
.$$ 
Now the point is because of \eqref{Hopf type lemma for the second radial eigenfunction} we find that $d^j_R\in \calH^s_0(A_R)$. Indeed, By \cite{MS} we know that
\be\label{eq-fract-normal-deriv}
d^{1-s}(x)\n u_R\cdot\nu(x)=-s\psi_R(x)\qquad\forall\,\,x\in\partial A_{R},\quad\text{with}\quad \psi_R(x):=\lim_{A_{R}\ni y\to x}\frac{u_R(y)}{d^s(y)}.
\ee
Applying this to the inner and outer boundary of $A_R$ gives respectively
\be\label{eq-boundary-reg}
s\ov\psi_R(0)=\lim_{t\to 0^+}t^{1-s}\partial_r\ov w_R(t)\qquad\text{and}\qquad -s\ov\psi_R(1)=\lim_{t\to 1^-}(1-t)^{1-s}\partial_r\ov w_R(t),
\ee
from which we deduce that $v^j_R(x)=\partial_r\ov u_R(|x|)\frac{x_j}{|x|}<0$ near $\partial_{out}A_R$ whenever $x_j>0$ since $\ov\psi_R(1)>0$ by \eqref{Hopf type lemma for the second radial eigenfunction}, that is, $(v^j_R)^+1_{H^j_+}=0$ near $\partial_{out}A_R$. Hence, $d^j_R=0$ near $\partial_{out}A_R\cap H^j_+$. By the same reasoning we show that $d^j_R=0$ near $\partial_{out}A_R\cap H^j_-$. Consequently $d^j_R=0$ near $\partial_{out}A_R$. Similarly, since $\ov\psi_R(0)<0$ by \eqref{Hopf type lemma for the second radial eigenfunction}, using the first identity in \eqref{eq-boundary-reg} we show that $d^j_R=0$ near $\partial_{in}A_R$. In conclusion we have $\supp(d^j_R)\subset\subset A_R$. By \cite[Lemma $2.2$]{Fetal}, to conclude that $d^j_R\in \cH^s_0(A_R)$, one simply needs to check that $d^j_R\in \calH^s_{loc}(A_R)$. For this we argue as in \cite[Lemma $3.5$]{Fetal}. In fact, since for all $\O'\subset\subset A_R$, we have that $\O'\cup\s_j(\O')\subset\subset A_R$ and $\big[(v^j_R)^+1_{H^j_+}\big]^2_{H^s(\O')}\leq \big[(v^j_R)^+1_{H^j_+}\big]^2_{H^s\big(\O'\cup \s_j(\O')\big)}$, we may assume without loss that $\O'$ is symmetric with respect to the reflection $\s_j$ across the hyperplane $\partial H^j_+$ and then argue as in \cite[Lemma $3.5$]{Fetal}. Having $d^j_R\in\calH^s_0(A_R)$ for all $j\in\{1,\cdots,N\}$, we can argue as in \cite{Fetal} to show that $u_R$ has Morse index greater than or equal $N+1$ and this contradicts the fact that $u_R$ is a second eigenfunction.
\end{proof}

\section{Maximization of the second eigenvalue}\label{sec-6}
We first start with some preliminaries. In the following, let $\tau\in(0,1)$ be fixed. For simplicity, we let $\O_a:=B_1(0)\setminus \ov{B_\t(ae_1)}$ for all $a\in(-1+\t,1-\t)$ and define

\be\label{first-anti-eigenvalue}
\l_{1,s}^-(a):=\l_{1,s}^-(\Omega_a)=\inf\left\lbrace  \frac{\frac{b_{N,s}}{2}\int_{\R^{2N}}\frac{(u(x)-u(y))^2}{|x-y|^{N+2s}}dxdy}{\int_{\O_a}u^2(x)dx}:u\in\calH^s_0(\O_a), u\circ \s_N=-u\right\rbrace ,
\ee
where $\s_N$ is the reflection with respect to the hyperplane $\partial H^N_+:=\{x\in\R^N:x_N=0\}$. It is a standard fact that the infimum in \eqref{first-anti-eigenvalue} is achieved by some $u$ which solves the equation $\Ds u=\l_{1,s}^-(a)u$ in $\O_a$ in the weak sense. We recall the following properties of minimizers of \eqref{first-anti-eigenvalue}.

\begin{lemma}[Theorem 1.2, \cite{SS}]\label{properties-of-first-anti-eigenfunction}The function $u$ achieving the infimum in \eqref{first-anti-eigenvalue} is unique (up to a multiplicative constant) and it is of one sign in $\O^+_a:=\O_a\cap \{x\in\R^N:x_N>0\}$. Without loss we may assume $u>0$ in $\O^+_a$.
\end{lemma}
Using the variational characterization 
$$
\l_{2,s}(a):=\l_{2,s}(\O_a)=\inf\left\lbrace  \frac{\frac{b_{N,s}}{2}\int_{\R^{2N}}\frac{(u(x)-u(y))^2}{|x-y|^{N+2s}}dxdy}{\int_{\O_a}u^2(x)dx}:u\in\calH^s_0(\O_a), \int_{\O_a}u(x)\phi_1(x)\,dx=0\right\rbrace,
$$
and recalling that the first eigenfunction $\phi_1$ of $\O_a$ is symmetric with respect to $\s_N$, one easily checks that
\be\label{upper-bound-lambda-2-s}
\l_{2,s}(a)\leq \l_{1,s}^-(a)\qquad\forall\,a\in(-1+\t,1-\t)
\ee
The following observation is of key importance for the proof of Theorem \ref{thm:max-prob}.

\begin{lemma}\label{key-identity}
Assume a second eigenfunction of the annulus $\O_0:=B_1(0)\setminus\ov{B_\t(0)}$ cannot be radial. Then we have 
\be\label{eq-6.3}
\l_{2,s}(0)=\l_{1,s}^-(0).
\ee
\end{lemma}
\begin{proof}
By \eqref{upper-bound-lambda-2-s}, it remains to prove $\l_{2,s}(0)\geq \l_{1,s}^-(0)$. For that we let $\l^{(N+2l)}_j$, $l\geq 0$ and $j\geq 1$ be the set of radial eigenvalues of the problem $\Ds u=\l u$ in $\O_0\subset \R^{N+2l}$ and $u=0$ in $\R^{N+2l}\setminus \O_0$. By \cite[Proposition 3(ii)]{Dyda-2}, we know that the first eigenfunction $\phi_1^{(N+2)}(|\cdot|)$ gives rise to $N$-linearly independent eigenfunctions $x_1\phi_1^{(N+2)}(|x|),\cdots, x_N\phi_1^{(N+2)}(|x|)$ to the problem $\Ds u=\l u$ in $\O_0\subset \R^N$ and $u=0$ in $\R^N\setminus \O_0$ with the same eigenvalue $\l^{(N+2)}_1$. Consequently, the second eigenvalue $\l_{2,s}(0)$ of $\O_0$ is given by $\l_{2,s}(0)=\min(\l^{(N+2)}_1,\l^{(N)}_2)$. To see the latter equality note that on the one hand we clearly have $\l_{2,s}(0)\leq\min(\l^{(N+2)}_1,\l^{(N)}_2)$. But on the other hand, the set $\{\lambda_j^{(N+2p)}\}_{j\in \N, p\in\N_0}$ contains all the eigenvalues in $\O_0\subset \R^N$ by \cite[Proposition 1]{Dyda} in combination with \cite[Proposition 3]{Dyda-2}. Since moreover $\lambda_j^{(N+2p)}$ is nondecreasing in $p$ and $j$, it follows that $\l_{2,s}(0)\geq\min(\l^{(N+2)}_1,\l^{(N)}_2)$ and thus equality must hold. Now, since by assumption a second eigenfunction cannot be radial, we deduce that $\l_{2,s}(0)=\l^{(N+2)}_1$. Therefore the eigenspace corresponding to $\l_{2,s}(0)$ is spanned by the functions $x\mapsto x_j\phi_1^{(N+2)}(|x|)$ with $j=1,\cdots,N$. Using $x_N\phi^{(N+2)}_1(|\cdot|)$ as a test function in \eqref{first-anti-eigenvalue} we get $\l_{2,s}(0)\geq \l_{1,s}^-(0)$ as wanted.
\end{proof}

\begin{proposition}\label{prop-variation-lambda-1-s-minus}
The mapping $(-1+\t,1-\t)\to\R_+,\; a\mapsto \l_{1,s}^-(a)$ is differentiable. Moreover, $(\l_{1,s}^-)'(0)=0$ and $(\l_{1,s}^-)'(a)<0$ for all $a\in(0,1-\t)$.
\end{proposition}
\begin{proof}Fix $a\in(-1+\t,1-\t)$ and let $\rho\in C^\infty_c(B_1)$ so that $\rho\equiv 1$ near $B_\t(ae_1)$ and $\rho\circ \s_N=\rho$. For any $\eps\in(-\eps_0,\eps_0)$, we consider the map $\Phi_\eps:\R^N\to\R^N, x\mapsto \Phi_\eps(x)=x+\eps\rho(x)e_1$. By making $\eps_0$ sufficiently small if necessary, one may assume that $\Phi_\eps:\R^N\to\R^N$ is a global diffeomorphism for all $\eps\in(-\eps_0,\eps_0)$. Moreover, it is easily verifiable that $\Phi_\eps(\O_a)=\O_{a+\eps}$ and hence $\l_{1,s}^-(a+\eps)=\l_{1,s}^-\big(\Phi_\eps(\O_a)\big)$. Now it suffices to prove that the function $\eps\mapsto\l_{1,s}^-\big(\Phi_\eps(\O_a)\big)$ is differentiable at $0$ and this is a consequence of the simplicity of $\l^-_{1,s}(a)$ stated in Lemma \ref{properties-of-first-anti-eigenfunction}. Indeed, let $u$ be the unique normalized minimizer corresponding to $\l_{1,s}^-(a)$, then the function $$w_\eps:=\frac{u\circ \Phi_\eps^{-1}-[u\circ \Phi_\eps^{-1}]\circ \s_N}{2}$$ is admissible in the variational characterization of $\l_{1,s}^-(a+\eps)$ and therefore 
$$
\l_{1,s}^-(a+\eps)\leq \frac{\frac{c_{N,s}}{2}\int_{\R^{2N}}\frac{(w_\eps(x)-w_\eps(y))^2}{|x-y|^{N+2s}}dxdy}{\int_{\O_{a+\eps}}w_\eps^2(x)dx}=:m(\eps).
$$
Since $m(0)=\l_{1,s}^-(a)$, we deduce that 
\be\label{eq-estimate-directional-deriv-of-lambda-1-s-minus}
\limsup_{\eps\to 0^+}\frac{\l_{1,s}^-(a+\eps)-\l_{1,s}^-(a)}{\eps}\leq \frac{d}{d\eps}_{\Big|\eps=0}m(\eps).
\ee
By a simple change of variables and using that $\phi_\eps\circ \s_N=\s_N\circ\phi_\eps$, one obtains
\be\label{eq-deriv-m-eps}
\frac{d}{d\eps}_{\Big|\eps=0}m(\eps)=\int_{\R^{2N}}(u(x)-u(y))^2K_X(x,y)dxdy+\l_{1,s}^-(a)\int_{\O_a}u^2(x)\div X(x)dx,
\ee
where 
\begin{align}
K_X(x,y)=\frac{b_{N,s}}{2}\left\lbrace \bigl(\div\,X(x) + \div\,X(y)\bigr)- (N+2s)\frac{ \bigl(X(x)-X(y)\bigr)\cdot (x-y)}{|x-y|^2}\right\rbrace \frac{1}{|x-y|^{N+2s}} 
\end{align}
and $X=\rho e_1\in C^{\infty}_c(B_1,\R^N)$. Since $u$ solves weakly the equation 
$$
\left\{\begin{aligned} &\Ds u=\l_{1,s}^-(a)u &&\text{in}\quad\O_a;\\
&u\in\cH^s_0(\O_a),&&\end{aligned}\right.
$$
the standard regularity theory see e.g \cite{RX} gives $u\in L^\infty(\O_a)\cap C^1_{loc}(\O_a)$ and therefore $u$ satisfies the assumptions of \cite[Theorem $1.2$]{ST}. Consequently 
\begin{align*}
&\int_{\R^{2N}}(u(x)-u(y))^2K_X(x,y)dxdy\\
&=\G^2(1+s)\int_{\partial\O_a}(u/d^s)^2X\cdot\nu\,dx+2\int_{\O_a}\n u\cdot X\Ds u\,dx\\
&=\G^2(1+s)\int_{\partial\O_a}(u/d^s)^2X\cdot\nu\,dx+2\l_{1,s}^-\int_{\O_a}u\n u\cdot X\,dx\\
&=\G^2(1+s)\int_{\partial\O_a}(u/d^s)^2 X\cdot\nu\,dx-\l_{1,s}^-(a)\int_{\O_a}u^2(x)\div X(x)\,dx
\end{align*}
Here, $\nu$ is the outer unit normal to the boundary. Plugging this into \eqref{eq-deriv-m-eps} and recalling \eqref{eq-estimate-directional-deriv-of-lambda-1-s-minus}, we conclude that
\begin{align}\label{eq-limsup-estimate}
\limsup_{\eps\to 0^+}\frac{\l_{1,s}^-(a+\eps)-\l_{1,s}^-(a)}{\eps}&\leq \G^2(1+s)\int_{\partial\O_a}(u/d^s)^2X\cdot\nu\,dx\nonumber\\
&=\G^2(1+s)\int_{\partial B_\t(ae_1)}(u/d^s)^2 e_1\cdot\nu\,dx
\end{align}
Next, let $\eps_n$ be a subsequence such that 
$$
\liminf_{\eps\to 0^+}\frac{\l_{1,s}^-(a+\eps)-\l_{1,s}^-(a)}{\eps}=\lim_{n\to\infty}\frac{\l_{1,s}^-(a+\eps_n)-\l_{1,s}^-(a)}{\eps_n}
$$
Let $u_{\eps_n}$ be the (normalized) minimizer of $\lambda_{1,s}^-(a+\eps_n)$, then $v_n:=u_{\eps_n}\circ \Phi_{\eps_n}\in\calH^s_0(\O_a)$ satisfies $v_n\circ \sigma_N=-v_n$. Moreover, we have for $n$ sufficiently large
\begin{align*}
C(N,s)\geq \frac{2}{b_{N,s}}\lambda_{1,s}^-(a+\eps_n)&= \int_{\R^{2N}}\frac{(u_{\eps_n}(x)-u_{\eps_n}(y))^2}{|x-y|^{N+2s}}dxdy\\
&=\int_{\R^{2N}}\frac{(v_{n}(x)-v_{n}(y))^2}{|\Phi_{\eps_n}(x)-\Phi_{\eps_n}(y)|^{N+2s}}\textrm{Jac}_{\Phi_{\eps_n}}(x)\textrm{Jac}_{\Phi_{\eps_n}}(y) dxdy\\
& \geq \ov C(N,s) \int_{\R^{2N}}\frac{(v_{n}(x)-v_{n}(y))^2}{|x-y|^{N+2s}}dxdy,
\end{align*}
where in the last estimate we have used \cite[estimate $(3.6)$]{Setal}. Hence, we may assume after passing to a subsequence still denoted by $n$, that
$$
v_n\to v_0\quad \text{weakly in $\calH^s_0(\O_a)$}, \quad v_n\to v_0\quad \text{in $L^2(\O_a)$},\quad \text{and}\quad v_n\to v_0\quad \text{ a.e in $\O_a$}
$$
By the pointwise converge, we check that $v_0\circ \sigma_N=-v_0$. Using $v_0$ as a test function in the variational characterization of $\l_{1,s}^-(a)$ and arguing as in \cite[Lemma $4.5$]{Setal}, we obtain
\be\label{eq-liminf-estimate}
\liminf_{\eps\to 0^+}\frac{\l_{1,s}^-(a+\eps)-\l_{1,s}^-(a)}{\eps}=\lim_{n\to\infty}\frac{\l_{1,s}^-(a+\eps_n)-\l_{1,s}^-(a)}{\eps_n}
\geq \G^2(1+s)\int_{\partial B_\t(ae_1)}(u/d^s)^2 e_1\cdot\nu\,dx
\ee
Combining \eqref{eq-limsup-estimate} and \eqref{eq-liminf-estimate} yields
\be\label{eq-right-limit}
\partial^+_\eps\l_{1,s}^-(a+\eps):=\lim_{\eps\to 0^+}\frac{\l_{1,s}^-(a+\eps)-\l_{1,s}^-(a)}{\eps}= \G^2(1+s)\int_{\partial B_\t(ae_1)}(u/d^s)^2 e_1\cdot\nu\,dx. 
\ee
Applying \eqref{eq-right-limit} to $\eps\mapsto\l_{1,s}^-(a-\eps)$ gives
\begin{align}
\lim_{\eps\to 0^-}\frac{\l_{1,s}^-(a+\eps)-\l_{1,s}^-(a)}{\eps}=-\partial^+_\eps\l_{1,s}^-\big(a-\eps\big)=\G^2(1+s)\int_{\partial B_\t(ae_1)}(u/d^s)^2 e_1\cdot\nu\,dx.
\end{align}
Consequently 
\be\label{shape deriv-of-the-first-antisymmetric-eigenvalue}
(\l_{1,s}^-)'(a)=\lim_{\eps\to 0}\frac{\l_{1,s}^-(a+\eps)-\l_{1,s}^-(a)}{\eps}=\G^2(1+s)\int_{\partial B_\t(ae_1)}(u/d^s)^2 e_1\cdot\nu\,dx,
\ee
for all $a\in(-1+\t,1-\t)$. Since $\l_{1,s}^-(a)=\l_{1,s}^-(-a)$, we have $(\l_{1,s}^-)'(0)=0$. Next, fix $a\in(0,1-\t)$ and let $H_a:=\{x\in\R^N:x_1>a\}$. Moreover, let us denote by $\s_a$ the reflection with respect to the hyperplane $\partial H_a$. For simplicity, we let $\ov u:=u\circ \s_a$, $w:=\ov u-u$ and $B^a:=B_\t(ae_1)$. It is not difficult to check that $w$ solves
\be\label{eq:weak-antisymmetric-eq-w}
\Ds w = \l^-_{1,s}(a)w\qquad\text{in}\qquad \Theta:=H_a\cap\O^+_a.
\ee
Moreover, by Lemma \ref{properties-of-first-anti-eigenfunction} we have 
\be\label{boundary-condiyion-w}
w\geq 0\quad\text{in}\quad H_a\cap H_+^N\setminus \Theta,
\ee
By \cite[Proposition $2.3$]{SS}, the maximum principle for doubly antisymmetric functions, we get that $w\geq 0$. Consequently $w>0$ in $\Theta$ by the strong maximum principle \cite[Proposition $2.4$]{SS} since $w>0$ in $\s_a(B_1)\cap H_a\cap H^N_+\setminus\Theta$.  Finally, by \cite[Proposition $2.4$]{SS}, we deduce that
\be\label{boundary-fractional-normal-deriv-of-w}
0<\frac{w}{d^s}=\frac{\ov u}{d^s}-\frac{u}{d^s}\qquad\text{on}\qquad \partial B^a\cap H^N_+\cap H_a=\partial B^a\cap H_a\cap\O^+_a.
\ee
It is also clear that 
\be\label{boundary-fractional-normal-deriv-of-u}
\frac{u}{d^s}\geq 0\qquad\text{on}\qquad \partial B^a\cap H_a\cap\O^+_a.
\ee
Now using the fact that $\s_N\Big(\partial B^a\cap H_a\cap\O^+_a\Big)=\partial B^a\cap H_a\cap\O^-_a$, $\s_N(\nu\cdot e_1) = \nu\cdot e_1$ and $\s_N\circ \s_a = \s_a\circ \s_N$, we get
\begin{align}\label{eq:rewriting-shape-deriv-of-lambda-minus}
&\int_{\partial B^a}\big(\frac{u}{d^s}\big)^2\nu\cdot e_1dx= \int_{\partial B^a\cap H_a}\Big[\big(\frac{u}{d^s}\big)^2-\big(\frac{u\circ \s_a}{d^s}\big)^2\Big]\nu\cdot e_1dx\nonumber\\
&=\int_{\partial B^a\cap H_a\cap\O^+_a}\Big[\big(\frac{u}{d^s}\big)^2-\big(\frac{\ov u}{d^s}\big)^2\Big]\nu\cdot e_1dx+\int_{\partial B^a\cap H_a\cap\O^-_a}\Big[\big(\frac{u}{d^s}\big)^2-\big(\frac{\ov u}{d^s}\big)^2\Big]\nu\cdot e_1dx\nonumber\\
&=\int_{\partial B^a\cap H_a\cap\O^+_a}\Big[\big(\frac{u}{d^s}\big)^2-\big(\frac{\ov u}{d^s}\big)^2\Big]\nu\cdot e_1dx+\int_{\partial B^a\cap H_a\cap\O^+_a}\Big[\big(\frac{u\circ \s_a}{d^s}\big)^2-\big(\frac{\ov u\circ \s_a}{d^s}\big)^2\Big]\nu\cdot e_1dx\nonumber\\
&=2\int_{\partial B^a\cap H_a\cap\O^+_a}\Big[\big(\frac{u}{d^s}\big)^2-\big(\frac{\ov u}{d^s}\big)^2\Big]\nu\cdot e_1dx\nonumber\\
&=-2\int_{\partial B^a\cap H_a\cap\O^+_a}\frac{w}{d^s}\Big(\frac{u}{d^s}+\frac{\ov u}{d^s}\Big)\nu\cdot e_1dx.
\end{align}
Combining \eqref{shape deriv-of-the-first-antisymmetric-eigenvalue}, \eqref{boundary-fractional-normal-deriv-of-w}, \eqref{boundary-fractional-normal-deriv-of-u}, and \eqref{eq:rewriting-shape-deriv-of-lambda-minus} we conclude that 
$$
(\l_{1,s}^-)'(a)= \G^2(1+s)\int_{\partial B^a}\big(u/d^s\big)^2\nu\cdot e_1dx<0\qquad\forall\,a\in(0,1-\t),
$$
as wanted.
\end{proof}
We are now ready to complete the proof of Theorem \ref{thm:max-prob}.
\begin{proof}[Proof of Theorem \ref{thm:max-prob}]
Since the problem is rotational and translation invariant, we may consider without loss of generality domains of the form  
\[\O_a:=B_1\setminus\ov{B_\t(ae_1)},\qquad\text{for}\quad a\in [0,1-\t).\]
By  \eqref{upper-bound-lambda-2-s}, Lemma \ref{key-identity}, and Proposition \ref{prop-variation-lambda-1-s-minus}, it immediately follows that
\be\label{eq:4.41}
\l_{2,s}(a)\leq \l^-_{1,s}(a)<\l^-_{1,s}(0)=\l_{2,s}(0)
\ee
for all $a\in (0,1-\t)$. The proof is finished.
\end{proof}

\bibliographystyle{amsplain}

\end{document}